\documentclass[11pt]{amsart}
\usepackage{graphicx,amsfonts,amssymb,amsmath,amsthm,url,
  verbatim,amscd}
  \usepackage{color}
\usepackage[usenames,dvipsnames]{xcolor}
\usepackage[normalem]{ulem}
\usepackage{pdfsync}
\usepackage{booktabs}
\usepackage[hyperfootnotes=false, colorlinks, citecolor=RoyalBlue,
urlcolor=blue, linkcolor=blue ]{hyperref}

\theoremstyle{plain}
\newtheorem*{theorem*}{Theorem}
\newtheorem{theorem}  {Theorem}    [section]
\newtheorem{lemma}      [theorem]{Lemma}
\newtheorem{corollary}  [theorem]{Corollary}
\newtheorem{proposition}[theorem]{Proposition}
\newtheorem{conjecture} {Conjecture}

\newtheorem{remark}  [theorem] {Remark}
\theoremstyle{definition}

\newtheorem{definition} [theorem]{Definition}

\allowdisplaybreaks

\renewcommand{\H}{\mathbb H}
\newcommand{\A}{{\mathbb A}}

\newcommand{\f}{{\mathbf f}}
\newcommand{\F}{{\mathcal F}}
\newcommand{\J}{{\mathcal J}}
\newcommand{\Q}{{\mathbb Q}}
\newcommand{\Z}{{\mathbb Z}}

\newcommand{\R}{{\mathbb R}}

\newcommand{\bs}{\backslash}

\newcommand{\p}{\mathfrak p}

\renewcommand{\P}{\mathcal P}
\newcommand{\OF}{{\mathfrak o}}
\newcommand{\GL}{{\rm GL}}

\newcommand{\SL}{{\rm SL}}
\newcommand{\SO}{{\rm SO}}

\newcommand{\sgn}{{\rm sgn}}

\newcommand{\ur}{{\rm ur}}

\newcommand{\mat}[4]{{\setlength{\arraycolsep}{0.5mm}\left[
      \begin{array}{cc}#1&#2\\#3&#4\end{array}\right]}}

\def\PSL{\operatorname{PSL}}

\def\SL{\operatorname{SL}}

\def\GL{\operatorname{GL}}

\renewcommand{\Im}{\mathrm{Im}}
\def\eps{\varepsilon}

\begin{document}

\bibliographystyle{plain}

\title{Hybrid sup-norm bounds for Maass newforms of powerful level}
\thanks{The author is supported by EPSRC grant EP/L025515/1.}

\author{Abhishek Saha}
\address{Department of Mathematics\\
  University of Bristol\\
  Bristol BS81TW\\
  UK} \email{abhishek.saha@bris.ac.uk}

\begin{abstract}Let $f$ be an $L^2$-normalized Hecke--Maass cuspidal newform of level $N$, character $\chi$
and Laplace eigenvalue $\lambda$. Let $N_1$ denote the smallest integer such that $N|N_1^2$ and $N_0$ denote the largest integer such that $N_0^2 |N$. Let $M$ denote the conductor of $\chi$ and define $M_1= M/\gcd(M,N_1)$. We prove the bound  $\|f\|_\infty \ll_{ \eps}   N_0^{1/6 + \eps} N_1^{1/3+\eps} M_1^{1/2} \lambda^{5/24+\eps},$ which generalizes and strengthens previously known upper bounds for $\|f\|_\infty.$

This is the first time a hybrid bound (i.e., involving both $N$ and $\lambda$) has been established for $\|f\|_\infty $  in the case of non-squarefree $N$. The only previously known bound in the non-squarefree case was in the $N$-aspect; it had been shown by the author that $\|f\|_\infty \ll_{\lambda, \eps}   N^{5/12+\eps}$ provided $M=1$. The present result significantly improves the exponent of $N$ in the above case. If $N$ is a squarefree integer, our bound reduces to $\|f\|_\infty \ll_\eps N^{1/3 + \eps}\lambda^{5/24 + \eps}$, which was previously proved by Templier.

The key new feature of the present work is a systematic use of $p$-adic representation theoretic techniques and in particular a detailed study of Whittaker newforms and matrix coefficients for $\GL_2(F)$ where $F$ is a local field.
\end{abstract}

\maketitle

\section{Introduction}\label{s:introduction}

\subsection{The main result}Let $f$ be a Hecke--Maass cuspidal newform on the upper half plane of weight 0, level $N$, character $\chi$, and Laplace eigenvalue $\lambda$. We normalize the volume of $Y_0(N)$ to be equal to 1 and assume that $\langle f, f\rangle := \int_{Y_0(N)} |f(z)|^2 dz = 1$. The problem of bounding the sup-norm $\|f\|_\infty:=\sup_{z\in Y_0(N)} |f(z)|$ in terms of the parameters $N$ and $\lambda$ is interesting from several points of view (quantum chaos, spectral geometry, subconvexity of $L$-functions, diophantine analysis) and has been much studied recently. For \emph{squarefree} levels $N$, there were several results, culminating in the best currently known bound due to Templier~\cite{templier-sup-2}, which states that $$\|f\|_\infty \ll_\eps \lambda^{5/24 + \eps} N^{1/3 + \eps}.$$
The exponent $5/24$ above  for $\lambda$ has stayed stable since the pioneering work of Iwaniec and Sarnak \cite{iwan-sar-85} (who proved $\|f\|_\infty \ll_\eps \lambda^{5/24 + \eps}$ in the case $N=1$), and it will likely require some key new idea to improve it. The exponent $1/3$ for $N$ in the above bound also appears difficult to improve, at least for squarefree levels, as it seems that the method used so far, primarily due to Harcos and Templier~\cite{harcos-templier-1, harcos-templier-2, templier-sup, templier-sup-2} has been pushed to its limit. The purpose of the present paper is to show that the situation is very different for \emph{powerful} (non-squarefree) levels.

To state our result, we introduce a bit of notation. Let $N_1$ denote the smallest integer such that $N|N_1^2$. Let $N_0$ be the largest integer such that $N_0^2 |N$. Thus $N_0$ divides $N_1$ and $N=N_0N_1$.\footnote{If $N$ has the prime factorization $N = \prod_p p^{n_p}$, then $N_0$ has the prime factorization $N_0 = \prod_p p^{\lfloor {n_p/2} \rfloor}$ and $N_1$ has the prime factorization $N_1 = \prod_p p^{\lceil {n_p/2} \rceil}$.} Note that if $N$ is squarefree, then $N_1=N$ and $N_0=1$. On the other hand, if $N$ is a perfect square or if $N$ is highly powerful (a product of high powers of primes) then $N_1 \asymp N_0 \asymp \sqrt{N}$. Also, let $M$ be the conductor of $\chi$ (so $M$ divides $N$) and put $M_1= M/\gcd(M,N_1)$. Note that $M_1$ divides $N_0$, and in fact $M_1$ equals 1 if and only if $M$ divides $N_1$. We will refer to the complementary situation of $M_1 >1$ (i.e., $M \nmid N_1$) as the case when the character $\chi$ is \emph{highly ramified}.

We prove the following result, which generalizes and strengthens previously known upper bounds for $\|f\|_\infty.$

\begin{theorem*} (Theorem~\ref{t:globalmain}) We have $$\|f\|_\infty \ll_{ \eps}   N_0^{1/6 + \eps} N_1^{1/3+\eps} M_1^{1/2} \lambda^{5/24+\eps}.$$
\end{theorem*}

Thus in the squarefree case, our result reduces to that of Templier. However when $N_1 \asymp N_0 \asymp \sqrt{N}$, and $M|N_1$ (i.e., $\chi$ is not highly ramified), our result gives $\|f\|_\infty \ll_{ \eps}   N^{1/4 + \eps} \lambda^{5/24+\eps}.$ The exponent $1/4$ we obtain in this case is better than the exponent 1/3 in the squarefree case.

 We note that the only upper bound known before this for general (i.e., possibly non-squarefree) $N$ is due to the present author, and was proved only very recently~\cite{sahasuplevel}. It was shown that $$\|f\|_\infty \ll_{\lambda, \eps}  N^{5/12 + \eps}$$  when $\chi$ is trivial (no dependence on $\lambda$ was proved). The results of this paper not only substantially improve those of~\cite{sahasuplevel} but also use quite different methods. We believe that the approach we take in this paper, characterized by a systematic use of adelic language and local representation-theoretic techniques that separate the difficulties place by place, is the right one to take for powerful levels.

 As for the optimum upper bound for $\|f\|_\infty$, it is reasonable to conjecture that \begin{equation}\label{conjbd}\|f\|_\infty \ll_\eps N^{\eps} \lambda^{1/12+\eps},\end{equation} if $M_1 =1$  (i.e., provided $\chi$ is \emph{not} highly ramified). If true, \eqref{conjbd} is optimal as one can prove \emph{lower bounds} of a similar strength.\footnote{\textbf{Added in proof:} Recent work of the author with Yueke Hu suggests that \eqref{conjbd} may not hold in general in the case of powerful levels. This is due to the failure of Conjecture \ref{locallindelof} in certain cases, which we have recently discovered.} If $\chi$ is highly ramified, we cannot expect \eqref{conjbd} to hold, for reasons explained in \cite{sahasupwhittaker}. Roughly speaking, in the highly ramified case, the corresponding \emph{local Whittaker newforms} can have large peaks due to a conspiracy of additive and multiplicative characters. This leads to a lower bound for $\|f\|_\infty$ that is larger than $N^{\eps}$ in the $N$-aspect. This phenomenon was first observed by Templier \cite{templier-large} in the case when $\chi$ is maximally ramified ($M_1 = N_0$) and extended (to cover a much bigger range of $M_1$) in our paper \cite{sahasupwhittaker}. The fact that the factor $M_1^{1/2}$ is present in our main theorem above (giving worse upper bounds in the highly ramified case) fits nicely with this theme.

 In the table below we  compare the upper bound provided by this paper  with the lower bound provided by our paper  \cite{sahasupwhittaker}. We consider newforms of level $N = p^n$, $1 \le n \le 5$. The second column gives the possible values of $M$ in each case. The third, fourth and fifth columns give the corresponding values of $N_0$, $N_1$ and $M_1$, respectively. The sixth column gives the upper bound provided by the main theorem  of this paper and should serve as a nice numerical illustration of our result in the \emph{depth aspect} ($N=p^n$, $n \rightarrow \infty$). The seventh and final column gives the corresponding lower bound proved in Theorem 3.3 of \cite{sahasupwhittaker}. The difference between these last two columns reflects the gap in the state of our current knowledge. As the table makes clear, the larger upper bounds for highly ramified $\chi$ are often matched by larger lower bounds.

 Finally, we have coloured blue all the quantities on the last column that we (optimistically) conjecture to be in fact the \emph{true size} of $\|f\|_\infty$ (up to a factor of $(N\lambda)^\eps$) in those cases.

$$
\begin{array}{ccccccc}
  \toprule
   N  &M  &N_0  &N_1  &M_1  &\|f\|_\infty \ll_{\eps} N^{\eps} \lambda^{5/24 + \eps} \times \ldots   &\|f\|_\infty \gg_{ \eps} N^{-\eps} \lambda^{1/12 - \eps} \times \ldots \\
  \toprule
   p&1 \text{ or } p&1&p&1& N^{1/3 }&\color{blue}1\\
   \midrule
   p^2&1 \text{ or } p&p&p&1& N^{1/4 }&\color{blue}1\\
  \cmidrule{2-7}

   &p^2&p&p&p& N^{1/2  }
   &\color{blue}N^{1/4  }\\
  \midrule
  p^3&1 \text{ or } p \text{ or } p^2&p&p^2&1& N^{5/18  }&\color{blue}1\\
  \cmidrule{2-7}
  &p^3&p&p^2&p& N^{4/9  }
   &\color{blue}N^{1/6  }\\
  \midrule
  p^4&1 \text{ or } p \text{ or } p^2&p^2&p^2&1&  N^{1/4  }&\color{blue}1\\
  \cmidrule{2-7}
  &p^3&p^2&p^2&p& N^{3/8  }
   &\color{red}1\\
   \cmidrule{2-7}
  &p^4&p^2&p^2&p^2& N^{1/2  }
   &\color{blue}N^{1/4  }\\
   \midrule
  p^5&1 \text{ or } p \text{ or } p^2 \text{ or } p^3&p^2&p^3&1&  N^{4/15  }&\color{blue}1\\
  \cmidrule{2-7}
  &p^4&p^2&p^3&p& N^{11/30  }
   &\color{blue}N^{1/10  }\\
   \cmidrule{2-7}
  &p^5&p^2&p^3&p^2&N^{7/15  }
   &\color{blue} N^{1/5  }\\
    \bottomrule
 \end{array}
$$

\subsection{Organization of this paper}
The remainder of Section \ref{s:introduction} is an extended introduction that explains some of the main features of our work. In Section \ref{sec:localcalcs}, which is the technical heart of this paper,  we undertake a detailed analytic study of $p$-adic Whittaker newforms and matrix coefficients for representations of $\GL_2(F)$ where $F$ is a nonarchimedean local field of characteristic 0. The two main results we prove are related to a) the support and average size of $p$-adic Whittaker newforms, b) the size of eigenvalues of certain matrix coefficients. These might be of independent interest. In Section \ref{s:global}, we prove the main result above. Perhaps surprisingly, and in contrast to our previous work \cite{sahasuplevel}, no counting arguments are needed in this paper beyond those supplied by Templier for the squarefree case. Also, in contrast to \cite{sahasuplevel}, we do not need any powerful version of the ``gap principle". Instead, we rely almost entirely on the $p$-adic results of Section \ref{sec:localcalcs}.

\subsection{Squarefree versus powerful levels}
The first bound for $\|f\|_\infty$ in the $N$-aspect was proved by Blomer and Holowinsky \cite{blomer-holowinsky} who showed that $\|f\|_\infty \ll_{\lambda, \epsilon} N^{\frac{216}{457} + \epsilon}$. They also proved the hybrid bound $\|f\|_\infty \ll (\lambda^{1/2} N)^{\frac{1}{2} - \frac{1}{2300}}.$  These results were only valid under the assumption that $N$ is squarefree. After that, there was fairly rapid progress (again only assuming $N$ squarefree) by Harcos and Templier~\cite{harcos-templier-1, harcos-templier-2, templier-sup, templier-sup-2}, culminating in the hybrid bound due to Templier described earlier. Note that the $N$-exponent in Templier's case is $1/3$, which may be viewed as the ``Weyl exponent", as it is a third of the way from the trivial bound of $N^{1/2 +\eps}$ towards the expected optimum bound\footnote{As mentioned earlier, this optimum bound is only expected to hold when $\chi$ is not highly ramified.} of $N^\eps$.

For a long time, there was no result at all when $N$ is not squarefree. Indeed, all the papers of Harcos and Templier rely crucially on using \emph{Atkin-Lehner operators} to move any point of $\H$ to a point of  imaginary part $\ge \frac1N$  (which is essentially equivalent to using a suitable Atkin--Lehner operator to move any cusp to infinity). This only works if $N$ is squarefree. In \cite{sahasuplevel}, the first (and only previous) result for Maass forms of  non-squarefree level was proved; assuming that $M=1$ we showed that $\|f\|_\infty \ll_{\lambda, \eps}  N^{5/12 + \eps}.$ A key new idea in \cite{sahasuplevel} was to look at the behavior of $f$ around \emph{cusps of width 1} and to formulate all the geometric and diophantine results around such a cusp. Apart from this, the overall strategy was not that different from the works of Harcos and Templier and the exponent of $5/12$ obtained was weaker than the exponent 1/3 for the squarefree case.

An initial indication that the exponent $1/3 $ in the $N$-aspect might be beaten for powerful levels was given by Marshall~\cite{marsh15}, who showed recently that  for a newform $g$ of level $N$ and trivial character on a \emph{compact arithmetic surface} (i.e., coming from a quaternion division algebra) the bound $$\|g\|_\infty \ll_{ \eps}  \lambda^{1/4 + \eps}N_1^{1/2 + \eps}$$ holds true. In particular, when $N$ is sufficiently powerful, this gives a ``sub-Weyl" exponent of $1/4 $ in the $N$-aspect. Marshall's proof does not work for the usual Hecke--Maass newforms $f$ on the upper-half plane of level $N$ that we consider in this paper (though it does work for certain shifts of these $f$ when restricted to a \emph{fixed compact} set). Finally, the main result of the present paper gives (when $\chi$ is not highly ramified) the bound $\|f\|_\infty \ll_{ \eps}   N_0^{1/6 + \eps} N_1^{1/3+\eps} \lambda^{5/24+\eps}$  which may be viewed as a strengthened analogue of Marshall's result for cusp forms on the upper-half plane.

 As indicated already, the powerful level case has been historically more difficult than the squarefree case.  It may thus seem surprising that in the powerful case,  we succeed in obtaining better exponents than in the squarefree case. However this seems to be a relatively common phenomenon. For example, for the related problem of quantum unique ergodicity in the level aspect, the known results in the squarefree case \cite{PDN-HQUE-LEVEL} give mass equidistribution with no power-savings but for powerful levels one obtains mass equidistribution with power savings \cite{NPS}. Again, for the problem of proving strong subconvexity bounds in the conductor aspect for Dirichlet $L$-functions, one only has a Weyl exponent $1/6$ when the conductor is squarefree, but Milicevic \cite{milisubconv} has shown an improved exponent of $.1645.. < 1/6$ for high prime powers. The results of this paper continue this surprising pattern (for which we do not attempt to give a general conceptual explanation).

 \subsection{Fourier expansions and efficient generating domains}

It seems worth noting explicitly the following interesting technical aspect of our work: the method of Fourier (Whittaker) expansion, once one chooses a good (adelic) \emph{generating domain}, leads to the rather strong bound $\|f\|_\infty \ll_{\eps}   M_1^{1/2}N_1^{1/2 + \eps} \lambda^{1/4+\eps}$. Note that this bound reduces to the ``trivial bound" when $N$ is squarefree, but is almost of the same strength (in the $N$-aspect) as our main theorem when $N$ is sufficiently powerful. In this subsection, we briefly  explain the ideas behind this.

It is best to work adelically here.
Let $\phi$ be the automorphic form associated to $f$, and let $g = g_\f g_\infty \in G(\A),$ where $g_\f$ denotes the finite part of $g$ and $g_\infty$ denotes the infinite component. Then $\|f\|_\infty = \sup_{g\in G(\A)} |\phi(g)|$. Because of the invariance properties for $\phi$, it suffices to restrict $g$ to a suitable  generating domain $D \subset  G(\A)$. Roughly speaking, $D$ can be any subset of $G(\A)$ such that the natural map from $D$ to $Z(\A) G(\Q) \bs G(\A) /\bar{K}$ is a surjection where $\bar{K}$ is a subgroup of $G(\A)$ generated by a set of elements under which $|\phi|$ is right-invariant.

The Whittaker expansion for $\phi$, which we want to exploit to bound $|\phi(g)|$, looks as follows: $$\phi(g) = \sum_{q \in \mathbb{Q}_{\neq 0}}
W_\phi(\mat{q}{}{}{1} g).$$ The above is an infinite sum, but two things make it tractable. First of all there is an integer $Q(g_\f)$, depending on $g_\f$, such that the sum is supported only on those $q$ whose denominator divides $Q(g_\f)$. Secondly, the sum decays very quickly after a certain point $|q| > T(g_\infty)$ due to the exponential decay of the Bessel function. The upshot is that \begin{equation}\label{whitintro}|\phi(g)| \ll \sum_{\substack{n \in \Z_{\neq 0}\\ |n|< Q(g_\f)T(g_\infty)}}
W_\phi(\mat{n/Q(g_\f)}{}{}{1} g).\end{equation}

The key quantity is the \emph{length} $Q(g_\f)T(g_\infty)$ of the sum above. Indeed, assuming Ramanujan type bounds on average for the local Whittaker newforms and using Cauchy-Schwartz, the expression \eqref{whitintro} leads to the inequality\footnote{Strictly speaking, this inequality is not completely accurate  as one has to add an (usually smaller) error term coming from peaks of the local Whittaker and $K$-Bessel functions.} $|\phi(g)| \ll_{\eps} (Q(g_\f)T(g_\infty))^{1/2 + \eps}$. The key point therefore, is to choose an efficient generating domain $D$ inside $G(\A)$, such that $\sup_{g \in D} Q(g_\f)T(g_\infty)$ is as small as possible.

Let us look at some examples. Suppose $\phi$ corresponds to a Hecke-Maass cusp form for $\SL_2(\Z)$. Then it is natural to take  $D$ to be the subset of $G(\A)$ consisting of the elements $g$ with $g_\f =1$ and $g_\infty= \mat{y}{x}{}{1}$ such that $-1/2 \le x \le 1/2$ and $y \ge \sqrt{3}/2$. In this case $Q(g_\f)=1$ and $T(g_\infty) = \lambda^{1/2}/y$, leading to the bound $|\phi(g)| \ll_{\eps} \lambda^{1/4 + \eps}$ as expected. Next, suppose $\phi$ corresponds to a newform of level $N$ where $N$ is squarefree. In this case one can include the Atkin-Lehner operators inside the symmetry group $\bar{K}$ above. Harcos and Templier showed that one can take $D$ to be the subset of $G(\A)$ consisting of the elements with $g_\f=1$ and  for which $g_\infty = \mat{y}{x}{}{1}$ such that $y \ge \sqrt{3}/(2N)$ (and some additional properties). For such an element, one again has $Q(g_\f)=1$ and $T(g_\infty) = \lambda^{1/2}/y$ leading to the bound $|\phi(g)| \ll_{\eps} (\lambda^{1/2}/y)^{1/2 + \eps} \ll   N^{1/2 + \eps}\lambda^{1/4 + \eps}$.

When $N$ is non-squarefree, it is not possible to construct a generating domain $D$ with a finite value of $\sup_{g \in D} T(g_\infty)$ for which all points have $g_\f=1$. Classically, this means that any fundamental domain (for the full symmetry group generated by $\Gamma_0(N)$ and the Atkin-Lehner operators) must touch the real line. The idea we used in our previous paper \cite{sahasuplevel} was to take the infinite part of $D$ essentially the same as in the squarefree case and take the finite part to be a certain nice subset of $\prod_{p|N}\GL_2(\Z_p)$. Classically, our choice of generating domain in \cite{sahasuplevel} corresponded to taking discs around cusps of width 1. Assuming $\chi=1$, this choice again gave $Q(g_\f)=1$ and $T(g_\infty) = \lambda^{1/2}/y$ leading to the same bound $|\phi(g)| \ll_{\eps} N^{1/2 + \eps}\lambda^{1/4 + \eps}$ as earlier. Thus, in all the above papers, the worst case bound obtained by the Whittaker expansion (i.e., for smallest $y$) was just the \emph{trivial bound} $N^{1/2 + \eps}\lambda^{1/4 + \eps}$ (also, all these papers restricted to $\chi=1$.)

In this paper we choose a somewhat different generating domain from that of \cite{sahasuplevel}.  For simplicity, we describe this domain here in the special case when $N=p^{2n_0}$, $M=p^m$, for some prime $p$ and some non-negative integers $n_0$, $m$. Take $D$ to consist of the elements $g_pg_\infty$ where $g_\infty = \mat{y}{x}{}{1}$ with $y \ge \sqrt{3}/2$ and $g_p \in \GL_2(\Z_p)\mat{p^{n_0}}{}{}{1}$. It is easy to prove this is a generating domain. The difficulties lie in computing $Q(g_\f)$ and in proving that the required Ramanujan type bounds on average hold. These key  technical local results involve intricate calculations that take up a good part  of Section \ref{sec:localcalcs}. We are able to prove that $\sup_{g \in D}Q(g_\f)= p^{\max(m,n_0)} = M_1\sqrt{N}$. Also, $T(g_\infty) = \lambda^{1/2}/y$ as usual. This leads to the surprisingly strong Whittaker expansion bound of $|\phi(g)| \ll_{\eps} M_1^{1/2} N^{1/4 + \eps}\lambda^{1/4 + \eps}$ in this case. Classically, the generating domain described above corresponds to taking discs around the cusps of the group $$\Gamma_0(p^{n_0}, p^{n_0}) = \{\mat{a}{b}{c}{d} \in \SL_2(\Z): p^{n_0}|b, \ p^{n_0}|c \}.$$ We remark here that the function $f'(z):=f(z/p^{n_0})$ is a Maass form for $\Gamma_0(p^{n_0}, p^{n_0})$.

When $N$ is not a perfect square, the generating domain we actually use is slightly different than described above. Roughly speaking, we exploit the existence of Atkin-Lehner operators at primes that divide $N$ to an odd power. This does not change the value of $\sup_{g \in D}  Q(g_\f)T(g_\infty)$ and so does not really affect the Whittaker expansion analysis; however it makes it easier to count lattice points for amplification (described in the next subsection).  In any case, the Whittaker expansion bound we prove ultimately (see Section \ref{s:fourierglobal}) is  $|\phi(g)| \ll_{\eps} (N_0M_1\lambda^{1/2}/y)^{1/2 + \eps}$ where $y \ge N_0/N_1$, leading to the worst case bound of $|\phi(g)| \ll_{\eps} M_1^{1/2} N_1^{1/2 + \eps}\lambda^{1/4 + \eps}$. This, as mentioned earlier, is essentially of the same strength (in the $N$-aspect) as our main theorem when $N$ is sufficiently powerful.

It bears repeating that the main tools used for  the above bound are local, relating to the representation theory of $p$-adic Whittaker functions. This supports the assertion of Marshall \cite{marsh15} that  $N_1^{1/2 + \eps}$ should be viewed as the correct \emph{local bound }in the level aspect (when $\chi$ is not highly ramified). Our analysis of these $p$-adic Whittaker functions also lead to other interesting questions. For example, one can ask for a sup-norm bound for these local Whittaker newforms, and in Conjecture \ref{conjbd}, we predict a Lindel\"of type bound when $\chi$ is not highly ramified (this conjecture was originally made in our paper \cite{sahasupwhittaker}). One of the key technical results in Section \ref{sec:localcalcs} essentially proves an \emph{averaged} version of this conjecture (this is the Ramanujan type bound on average alluded to earlier).

\subsection{The pre-trace formula and amplification}
Recall that our main theorem states that $\|f\|_\infty \ll_{ \eps}   N_0^{1/6 + \eps} N_1^{1/3+\eps} M_1^{1/2} \lambda^{5/24+\eps}$. As we have seen above, the method of Fourier (Whittaker) expansion gives us the bound $\|f\|_\infty \ll_{\eps}  N_1^{1/2 + \eps} M_1^{1/2}\lambda^{1/4+\eps}$ (with even better bounds when the relevant point on our generating domain has a large value for $y$) so we need to save a further factor of $(N_1/N_0)^{1/6}\lambda^{1/{24}}$. This is done by \emph{amplification}, whereby we choose suitable test functions at each prime to obtain a \emph{pre-trace formula} and then estimate its geometric side via some  point counting results due to Harcos--Templier \cite{harcos-templier-2} and Templier \cite{templier-sup-2}. The basic idea is that by choosing these local test functions carefully (constructing an amplifier) one should be able to boost the contribution of the newform $f$  to the resulting pre-trace formula. The details for this are given (in a fairly flexible adelic framework) in Sections \ref{s:amplglobal} -- \ref{s:conclusion}.

 The unramified local test functions that we use in this paper are standard and essentially go back to Iwaniec--Sarnak (the key point is to exploit a simple identity relating the eigenvalues for the Hecke operators $T(\ell)$ and $T(\ell^2)$). However, our ramified local test functions are very different from the papers of Harcos--Templier or our previous paper \cite{sahasuplevel}. In all those past papers, the ramified test functions had been simply chosen to be the characteristic functions of the relevant congruence subgroups. In contrast, we use a variant of the local test function used by Marshall in \cite{marsh15}. The main results about this test function are proved in Sections \ref{s:mcbeg} -- \ref{s:endmc}. Roughly speaking, it is (the restriction to a large compact subgroup of) the \emph{matrix coefficient} for a local vector $v'$ obtained by translating the local newform.
The key property of this test function is that its unique non-zero eigenvalue is fairly large (and $v'$ is an eigenvector with this eigenvalue).

  Our choice of test functions at ramified primes ensures that any pre-trace formula involving them averages over relatively few representations of level $N$. It may be useful to view this as a ramified analogue of the classical (unramified) amplifier. Indeed, the resulting ``trivial bound" obtained via the pre-trace formula (by choosing the unramified test functions
trivially) matches exactly (on compact subsets) with the strong local bounds obtained via Whittaker
expansion. This is an important point because it means that we only need to save a further factor of $(N_1/N_0)^{1/6}\lambda^{1/{24}}$ by putting in the unramified amplifier and counting lattice points. This is carried out in Section \ref{s:conclusion}.

 It is worth noting that we do not need any new counting results in this paper beyond those proved by Harcos and Templier. This is because the counting part of our paper is only concerned with the squarefree integer $N_1/N_0$. In particular, the role of amplification in this paper to improve the $N$ exponent  is relatively minor when $N$ is highly powerful (note that $N_1/N_0$ approaches a negligible power of $N$ as $N$ gets more powerful). Indeed, when $N$ is a perfect square ($N_1 = N_0$), all our savings in the $N$-aspect come from Whittaker expansion and we do not gain anything further by amplification.\footnote{However, we always gain a non-trivial savings in the $\lambda$ aspect via amplification.} In contrast, our paper \cite{sahasuplevel} had a relatively poor
bound coming from Whittaker expansion but we then saved a non-trivial power of $N$ via amplification.

 The technical reason why the method of amplification does not improve the $N$-aspect too much beyond our strong local bounds is that our ramified test functions have relatively large support. Consequently, we do not have many global congruences related to $N$, and  congruences are essential for savings via counting. More precisely, our ramified test functions are supported on the maximal compact subgroup at primes that divide $N$ to an even power, and supported on a (slightly) smaller subgroup at primes that divide $N$ to an odd power (it is the latter case that leads to the savings of $(N_1/N_0)^{1/6}$). If we were to reduce the support of our test functions further and thus force new congruences, the resulting savings via counting would be eclipsed by the resulting loss due to the fact that our pre-trace formula would now be averaging over more representations of level $N$. Somehow the ramified and unramified parts of the amplifier seem to work against each other and the key point is to strike the right balance.

 It would be an interesting and challenging problem to detect any additional cancellation on the geometric side of our pre-trace formula by going beyond counting lattice points and perhaps taking into account the \emph{phases }of the matrix coefficient used to construct the ramified test function. Such a result could potentially push the upper-bound for $\|f\|_\infty$ below $N^{1/4}$.

\subsection{Notations}
We collect here some general notations that will be used throughout this paper. Additional notations will be defined where they first appear in the paper.

Given two integers $a$ and $b$, we use $a|b$ to denote that $a$ divides $b$, and we use $a|b^\infty$ to denote that $a|b^n$ for some positive integer $n$. For any real number $\alpha$, we let $\lfloor \alpha \rfloor$ denote the greatest integer less than or equal to $\alpha$ and we let $\lceil \alpha \rceil$ denote the smallest integer greater than or equal to $\alpha$. The symbol $\A$ denotes the ring of adeles of $\Q$ and $\A_\f$ denotes the subset of finite adeles. For any two complex numbers $\alpha, z$, we let $K_\alpha(z)$ denote the modified Bessel function of the second kind.

The groups $\GL_2$, $\SL_2$, $\PSL_2$, $\Gamma_0(N)$ and $\Gamma_1(N)$ have their usual meanings. The letter $G$ always stands for the group $\GL_2$. If $H$ is any subgroup of $G$, and $R$ is any subring of $\R$, then $H(R)^+$ denotes the subgroup of $H(R)$ consisting of matrices with positive determinant.

 We let $\H = \{x + iy: x \in \R, \ y\in \R, \ y>0\}$ denote the upper half plane. For any $\gamma=\begin{pmatrix}a&b\\c&d\end{pmatrix}$ in $\GL_2(\R)^+$, and any $z \in \H$, we define $\gamma(z)$ or $\gamma z$ to equal $\frac{az+b}{cz+d}$. This action of $\GL_2(\R)^+$ on $\H$  extends naturally to the boundary of $\H$.

We say that a function $f$ on $\H$ is a Hecke--Maass cuspidal newform of weight 0, level $N$, character $\chi$ and Laplace eigenvalue $\lambda$ if it has the following properties:
      \begin{itemize}

     \item $f$ is a smooth real analytic function on $\H$.

     \item $f$ satisfies  $(\Delta + \lambda) f =
0$ where $\Delta := y^{-2}
(\partial_x^2 + \partial_y^2)$.

\item For all $\gamma = \mat{a}{b}{c}{d} \in \Gamma_0(N)$, $f(\gamma z) = \chi(d) f(z)$.

\item $f$ decays
rapidly at the cusps of $\Gamma_1(N)$.

\item  $f$ is orthogonal to all oldforms.

 \item $f$ is an eigenfunction of all  the Hecke and Atkin-Lehner operators.\footnote{Assuming the previous properties, this last property is equivalent to the weaker condition that $f$ is an eigenfunction of almost all Hecke operators.}

      \end{itemize}

The study of newforms $f$ as above is equivalent to the study of corresponding adelic newforms $\phi$ which are certain functions on $G(\A)$. For the details of this correspondence, see Remark \ref{r:adelic}.

 We use the notation
$A \ll_{x,y,z} B$
to signify that there exists
a positive constant $C$, depending at most upon $x,y,z$,
so that
$|A| \leq C |B|$.
 The symbol $\epsilon$ will denote a small positive quantity. The values of $\epsilon$ and that of the constant implicit in $\ll_{\epsilon, \ldots}$ may change from line to line.

\subsection{Acknowledgements} I would like to thank Edgar Assing, Simon Marshall, Paul Nelson, and Nicolas Templier for useful discussions and feedback.

\section{Local calculations}\label{sec:localcalcs}

\subsection{Preliminaries}\label{sec:2-notations}

We begin with fixing some notations that will be used throughout this section.
Let  $F$ be a non-archimedean local
  field  of characteristic zero whose
  residue field has cardinality $q$.
Let $\mathfrak{o}$ be its ring of integers,
and $\mathfrak{p}$ its maximal ideal.
Fix a generator $\varpi$ of $\mathfrak{p}$.
Let $|.|$ denote the absolute value
on $F$ normalized so that
$|\varpi| = q^{-1}$. For each $x \in F^\times$, let $v(x)$ denote the integer such that $|x| = q^{-v(x)}$.  For a non-negative integer $m$, we define the subgroup $U_m$ of  $\OF^\times$ to be the set of elements $x \in \OF^\times$ such that $v(x-1) \ge m$.

Let $G = \GL_2(F)$ and $K = \GL_2(\mathfrak{o})$.
For each integral ideal $\mathfrak{a}$ of $\mathfrak{o}$,
let
\[
K_0(\mathfrak{a}) = K \cap \begin{bmatrix}
  \mathfrak{o}  & \mathfrak{o}  \\
  \mathfrak{a}  & \mathfrak{o}
\end{bmatrix}, \quad
K_1(\mathfrak{a})
= K \cap \begin{bmatrix}
  1+ \mathfrak{a}  & \mathfrak{o}  \\
  \mathfrak{a}  & \mathfrak{o}
\end{bmatrix}, \quad K^0(\mathfrak{a}) = K \cap \begin{bmatrix}
  \mathfrak{o}  & \mathfrak{a}  \\
  \mathfrak{o}  & \mathfrak{o}
\end{bmatrix}.
\]

Write
\[
w = \begin{bmatrix}
  0 & 1 \\
  -1 & 0
\end{bmatrix},
\quad
a(y) = \begin{bmatrix}
  y &  \\
  & 1
\end{bmatrix},
\quad
n(x) = \begin{bmatrix}
  1 & x \\
  & 1
\end{bmatrix},
\quad z(t)
= \begin{bmatrix}
  t &  \\
  & t
\end{bmatrix}
\]
for $x \in F, \ y \in F^\times, \ t \in F^\times$.
Define subgroups
$N =
\{n(x):  x\in F \}$,
$A = \{a(y): y\in F^\times \}$,
$Z =\{ z(t):
t \in F^\times \}$, $B_1=NA$,
and $B = Z N A = G \cap
\left[
  \begin{smallmatrix}
    *&*\\
    &*
  \end{smallmatrix}
\right]$ of $G$.

We normalize Haar measures as
follows.
The measure $dx$ on the additive group $F$ assigns volume 1
to $\OF$, and transports to a measure on $N$.
The measure $d^\times y$ on the multiplicative group $F^\times$ assigns
volume 1 to $\OF^\times$,
and transports to measures on
$A$ and $Z$.
We obtain a left Haar measure $d_Lb$ on $B$ via
$d_L(z(u)n(x)a(y)) = |y|^{-1}\, d^\times u \, d x \, d^\times
y.$
Let $dk$ be the probability Haar measure on $K$.
The Iwasawa decomposition
$G = B K$ gives a left Haar measure $dg = d_L b \, d k$ on $G$.

For each irreducible admissible
representation $\sigma$ of $G$ (resp. of $F^\times$), we define $a(\sigma)$ to be the smallest non-negative integer such that $\sigma$ has a $K_1(\p^{a(\sigma)})$-fixed (resp. $U_{a(\sigma)}$-fixed) vector.

\subsection{Some matrix invariants}\label{s:repnew}

\emph{From now on, fix $\pi$ to be a generic
irreducible admissible unitary
representation of $G$. Let $n=a(\pi)$, and let $\omega_\pi$ denote the central character of $\pi$.}

It is convenient now to introduce some notation. Define

\begin{itemize}
\item $n_1 := \lceil \frac{n}{2} \rceil$,

\item $n_0:= n -n_1 = \lfloor \frac{n}{2} \rfloor$,

\item $m = a(\omega_\pi)$,

\item $m_1 = \max(0, m-n_1)$.
\end{itemize}
 Note that $m_1=0$ if and only if $m \le n_1$; this can be viewed as the case when $\omega_\pi$ is \emph{not} highly ramified.

 Next, for any $g \in G$, we define two integers $t(g)$ and $l(g)$ which depend on $g$ and $n$. Recall the disjoint double coset decomposition~\cite[Lemma 2.13]{sahasupwhittaker}:
\begin{equation}\label{e:coset}G = \bigsqcup_{t \in \Z}\bigsqcup_{0\le l \le n} \ \bigsqcup_{v \in \OF^\times/U_{\min(l, n-l)}} Z N a(\varpi^t)wn(\varpi^{-l}v) K_1(\p^{n}).\end{equation}
Accordingly, given any  matrix $g \in G$, we define $t(g)$ and $l(g)$ to be the unique integers such that

\begin{itemize}

\item $0 \le l(g) \le n$,

\item $g \in Z N a(\varpi^{t(g)})wn(\varpi^{-l(g)}v)K_1(\p^{n})$ for some $v \in \OF^\times$.

\end{itemize}

\begin{remark} It is illuminating to restate  these matrix invariants slightly differently. Let $g$ in $G$. The Iwasawa decomposition tells us that $g \in ZNa(y)k$ where $k = \mat{a}{b}{c}{d} \in K$. Then one can check that $l(g) = \min(v(c), n)$, and $t(g) = v(y) - 2l(g)$.
\end{remark}

 In the sequel, we will often consider matrices $g$ lying in the set $Ka(\varpi^{n_1})$. The next few lemmas explicate some key properties of this set.

\begin{lemma}\label{usefull}Suppose that $k \in K$ and $n$ is odd (so $n_1 = n_0+1$). Then

\begin{enumerate}

\item $l(ka(\varpi^{n_1})) \ge n_1$ if and only if $k \in N(\OF)K^0(\p)$.

 \item  $l(ka(\varpi^{n_1})) \le n_0$ if and only if $k \in wK^0(\p)$.
\end{enumerate}
\end{lemma}
\begin{proof} We first assume that $l(ka(\varpi^{n_1})) \ge n_1$ and prove that $k \in BK^0(\p)$. For brevity, put $l = l(ka(\varpi^{n_1}))$. So we can write $ka(\varpi^{n_1}) = bwn(\varpi^{-l}v) k'$ where $b \in B$, $k' \in K_1(\p^{n})$ and $n \ge l \ge n_1$. Therefore $k = b' w n(\varpi^{n_1-l}v) k_2$ where  $k_2 = a(\varpi^{n_1})k'a(\varpi^{-n_1}) \in K^0(\p^{n_1})$ and $b' \in B$. So to complete the proof that $k \in BK^0(\p)$, it suffices to check that there exists a matrix $b_2 \in B$ such that $b_2wn(\varpi^{n_1-l}v) \in K^0(\p)$. By explicit verification, $b_2=\left[\begin{smallmatrix}
\varpi^{n_1-l}v&1\\0&\varpi^{l-n_1}v^{-1}
\end{smallmatrix}\right]$ works.
Having proved that $k \in BK^0(\p)$, it follows immediately that $k \in B(\OF)K^0(\p) = N(\OF)K^0(\p)$.

The proof that $l(ka(\varpi^{n_1})) \le n_0$ implies $k \in wK^0(\p)$ is similar. The reverse implications follow using  $N(\OF)K^0(\p) \ \cap\ wK^0(\p)=\emptyset.$ 
\end{proof}

\begin{lemma}\label{lemma:ALnew}Suppose that $k \in K_0(\p)$, $n$ is odd, and $g \in \{1, \mat{}{1}{\varpi}{} \}.$ Then   $kgwa(\varpi^{n_1}) = k' a(\varpi^{n_1}) g' z$ where $k' \in K$, $l(k'a(\varpi^{n_1})) \le n_0$, $g' \in \{1, \mat{}{1}{\varpi^n}{}\}$, and $z \in Z$.
\end{lemma}

\begin{proof}If $g=1$, then $kgwa(\varpi^{n_1}) = w(w^{-1}kw) a(\varpi^{n_1})$. If $g=\mat{}{1}{\varpi}{}$, then $kgwa(\varpi^{n_1}) = w(w^{-1}kw) a(\varpi^{n_1})\mat{}{1}{\varpi^n}{} z(-\varpi^{n_1-n})$. The result follows from Lemma \ref{usefull} as $(w^{-1}kw) \in K^0(\p)$.
\end{proof}

\begin{lemma}\label{lemmatg}Suppose that $g \in K a(\varpi^{n_1})$. Then $t(g) = \min(n_1-2l(g), -n_1)$.
\end{lemma}
\begin{proof}This follows by an explicit computation similar to the proof of Lemma \ref{usefull}. We omit the details.
\end{proof}

\subsection{Our goal}

It may be worthwhile to declare at this point  the output from the rest of Section \ref{sec:localcalcs} that will be needed for our main theorem.

In Sections \ref{s:localwhitbeg} -- \ref{s:localwhit}, we will study the local Whittaker newform $W_\pi$, which is a certain function on $G$. Given a compact subset $\J$ of $G$, we are interested in the following questions:
\begin{enumerate}
\item For each $g \in \J$, provide a good upper bound for the quantity $\sup\{|y|: W_\pi(a(y)g) \neq 0\}$.\footnote{This is essentially the local analogue of the quantity $Q(g_\f)$ described in the introduction.}

 \item Prove an average Ramanujan-type bound for the function   $|W_\pi(a(y)g)|$ whenever $g\in \J$ and $W_\pi(a(y)g) \neq 0$.
\end{enumerate}
For our global applications, it will be useful to have the set $\J$ to be relatively large (so that we can create a generating domain out of it with a relatively small archimedean component) while also making sure that the supremum of the upper bound above (as $g$ varies in $\J$) is fairly small (so as to optimize the Whittaker expansion bound). We will choose $\J$ to equal the set $K a(\varpi^{n/2})$ if $n$ is even and equal to $\{g \in K a(\varpi^{n_1}): l(g) \le n_0\}$ if $n$ is odd. For this set $\J$ we will answer the two questions above in Proposition \ref{localwsupportfinal} below. This proposition will be of key importance for our global Whittaker expansion bound.

Next, in Sections \ref{s:mcbeg} -- \ref{s:endmc}, we will study a certain test function $\Phi_\pi'$.  This test function, viewed as a convolution operator, is essentially idempotent, and therefore has exactly one non-zero positive eigenvalue. In Proposition \ref{keymatrixprop}, we determine the size of this non-zero eigenvalue, and we also prove that $a(\varpi^{n_1}) \cdot W_\pi$ is an eigenvector with this eigenvalue. This proposition will be of key importance for our global bound coming from the amplified trace formula.

In view of the technical material coming up, it is worth emphasizing that Propositions \ref{localwsupportfinal} and \ref{keymatrixprop} are the \emph{only} results from the rest of Section \ref{sec:localcalcs} that will be used in Section \ref{s:global}.

\subsection{The Whittaker newform}\label{s:localwhitbeg}
Fix an additive character $\psi : F \rightarrow S^1$
with conductor $\mathfrak{o}$. Then $\pi$ can be realized as a unique subrepresentation of the space of functions $W$ on $G$ satisfying $W(n(x)g) = \psi(x)W(g)$. This is the Whittaker model of $\pi$ and will be denoted $\mathcal{W}(\pi,\psi)$.

\begin{definition}\label{defn:normalized-W-newform}
  The \emph{normalized Whittaker newform}
  $W_\pi $
  is
  the unique function in $\mathcal{W}(\pi,\psi)$ invariant under $K_1(\p^{n})$
  that satisfies $W_\pi(1)= 1$.
\end{definition}

The following lemma is well-known and so we omit its proof.

\begin{lemma}\label{whitsuppdiagtriv}Suppose that $W_\pi(a(y)) \neq 0$. Then $|y| \le 1$, i.e., $y \in \OF$.
\end{lemma}

\begin{lemma}[Proposition 2.28 of~\cite{sahasupwhittaker}]\label{atkin-lehnerwhit}Let $\tilde{\pi}$ denote the contragredient representation of $\pi$. Let $t \in \Z$, $0\le l \le n$, $v \in \OF^\times$, and assume that\footnote{There is no loss of generality in this assumption as we can always twist $\pi$ by a character of the form $| \ |^{ir}$ to ensure this.}  $\omega_\pi(\varpi)=1$. We have $$
W_{\tilde{\pi}}(a(\varpi^t)wn(\varpi^{-l}v)) = \eps(1/2,  \pi) \omega_\pi(v)\psi(-\varpi^{t+l}v^{-1})  W_\pi(a(\varpi^{t+2l-n})wn(-\varpi^{l-n}v)).$$

\end{lemma}

Define $g_{t,l,v}:= a(\varpi^t)wn(\varpi^{-l}v)$. Let $\tilde{X}$ denote the group of characters $\mu$ of $F^\times$ such that $\mu(\varpi)=1$.  For each  $\mu \in \tilde{X}$, and each $x \in F$, define the Gauss sum $G(x, \mu) = \int_{\OF^\times}\psi(xy) \mu(y) d^\times y.$

We will need two additional results for the results of the next subsection.
The first one is a key formula from~\cite{sahasupwhittaker}.

\begin{lemma}[Prop. 2.23 of \cite{sahasupwhittaker}] Assume that $\omega_\pi \in \tilde{X}$. We have $$W_\pi(g_{t,l,v}) = \sum_{\substack{\mu: a(\mu)\le l,\\ \mu \in \tilde{X}}} c_{t,l}(\mu) \mu(v),$$  where the coefficients  $c_{t,l}(\mu)$ can be read off from the following identity

\begin{equation}\label{basicid}\begin{split}&\varepsilon(\frac12, \mu\pi)\left(\sum_{t=-\infty}^\infty q^{(t+a(\mu\pi))(\frac12-s)} c_{t,l}(\mu)\right) L(s, \mu \pi)^{-1} \\&= \omega_\pi(-1)\left(\sum_{a=0}^\infty W_\pi(a(\varpi^a)) q^{-a(\frac12-s)}G(\varpi^{a-l}, \mu^{-1})\right) L(1-s, \mu^{-1}\omega_\pi^{-1}\pi)^{-1}.\end{split}\end{equation}
\end{lemma}

The next result deals with conductors of character twists. While the proof is quite easy, it involves a question that comes up frequently in such problems, see, e.g. Remark 1.9 of~\cite{NPS}.
\begin{lemma}\label{lemmatwist}Let $l \le n_0$ be a non-negative integer. For each character $\mu$ with $a(\mu)=l$, we have $a(\mu \pi) \le \max(n, \ l+m).$ Furthermore, for each $r \ge 0$,  $$\left| \{\mu \in \tilde{X} : a(\mu)=l, \ a(\mu \pi) = \max(n, \  l+m) - r\} \right| \le q^{l - \frac{r}2}.$$
\end{lemma}
\begin{proof}If $\pi$ is supercuspidal we have $l+m \le n$. Writing $\pi$ as a twist of a minimal supercuspidal, the result follows from Tunnell's theorem~\cite[Prop. 3.4]{tunnell78} on conductors of twists of supercuspidal representations. If $\pi$ is principal series, then it follows from the well-known formula $a(\chi_1 \boxplus \chi_2) = a(\chi_1) + a(\chi_2).$ If $\pi$ is a twist of the Steinberg representations, it follows from the formula $a(\chi \rm{St}) = \max(2 a(\chi), 1)$.
\end{proof}

\subsection{The support and average size of $W_\pi$}\label{s:localwhit}
In this subsection, we will prove an important technical result (Proposition~\ref{whitsuppsizeram}) about the size and support of $W_\pi$. This will then be combined with the results of the previous subsection to deduce Proposition \ref{localwsupportfinal} which will be needed for our global application.  To motivate all these results, we first recall a certain conjecture made in~\cite{sahasupwhittaker}.

\begin{conjecture}[Local Lindelof hypothesis for Whittaker newforms]\label{locallindelof}Suppose that $a(\omega_\pi) \le n_1$ (i.e., $m_1 = 0$). Then $$1 \ll \sup_{g \in G} |W_\pi(g)| \ll_\eps q^{n\eps}.$$
\end{conjecture}

This conjecture (originally stated as \cite[Conjecture 2]{sahasupwhittaker}) seems to be quite hard as it implies square-root cancelation in sums of twisted $\GL_2-\varepsilon$-factors. However, for the purpose of this paper, we can prove a bound  that is (at least) as strong as the above conjecture on \emph{average}. This is achieved by the second part of the next proposition, which generalizes some results obtained in~\cite[Section 2]{NPS}, which considered the special case $\omega_\pi=1$.

\begin{proposition}\label{whitsuppsizeram}The following hold. \begin{enumerate}
\item If $W_\pi(g) \neq 0$, then  $t(g) \ge -\max(2l(g), l(g)+m, n)$.

\item Suppose $t(g) =-\max(2l(g), l(g)+m, n) +r$ where $r \ge 0$. Then we have $$\left(\int_{v \in \OF^\times} \left|W_\pi( a(v)g) \right|^2 d^\times v\right)^{1/2} \ll q^{-r/4}.$$

\end{enumerate}
\end{proposition}

\begin{proof} By twisting $\pi$ with a character of the form $| \ |^{ir}$ if necessary (which does not change $|W_\pi|$), we may assume $\omega_\pi \in \tilde{X}$. Also assume $n \ge 1$ as the case $n=0$ is trivial. Because of the coset decomposition from earlier, we may further assume that $g=g_{t,l,v}:= a(\varpi^t)wn(\varpi^{-l}v))$. Finally, because of Lemma \ref{atkin-lehnerwhit}, we can assume (by replacing $\pi$ by $\tilde{\pi}$ if necessary) that $0 \le l \le n_0$. The desired result then is the following:

\begin{itemize}
\item Let $0 \le l \le n_0$. If $W_\pi(g_{t,l,v}) \neq 0$, then  $t \ge -\max(n, l +m)$. Further if $t =  -\max(n, l +m)+r$ where $r \ge 0$ then $$\left(\int_{v \in \OF^\times} \left|W_\pi(g_{t,l,v}) \right|^2 d^\times v\right)^{1/2} \ll q^{-r/4}.$$

\end{itemize}

 In the notation of~\eqref{basicid}, the above is equivalent to:

\textbf{Claim 1:} Let $0 \le l \le n_0$. If there exists $\mu\in \tilde{X}$ such that $a(\mu)\le l$ and $c_{t,l}(\mu) \neq 0$ then  $t \ge -\max(n, l +m)$. Further if $t = -\max(n, l +m)+r$ where $r \ge 0$ then $\sum_{\substack{\mu \in \tilde{X} \\ a(\mu)\le l}} \left|c_{t,l}(\mu) \right|^2  \ll q^{-r/2}.$

 Define the quantities $d_{t,l}(\mu)$ via the following identity (of polynomials in $q^{\pm s}$).

\begin{equation}\label{basicid2}
\varepsilon(\frac12, \mu\pi)\left(\sum_{t=-\infty}^\infty q^{(t+a(\mu\pi))(\frac12-s)} c_{t,l}(\mu)\right) L(s, \mu \pi)^{-1} = \left(\sum_{t=-\infty}^\infty q^{(t+a(\mu\pi))(\frac12-s)} d_{t,l}(\mu)\right) \end{equation}

Note that (for fixed $l$ and $\mu$)  $d_{t,l}(\mu)$ is non-zero for only finitely many $t$. Furthermore, $c_{t,l}(\mu) = \sum_{i=0}^\infty \alpha_i d_{t-i,l}(\mu)$ where $|\alpha_0| =1$ and $|\alpha_i| \ll q^{-i/2}$. (In fact, if $\pi$ is supercuspidal, $\alpha_i = 0$ for all $i>0$). Hence it suffices to prove Claim 1 above for the quantities $d_{t,l}(\mu)$ rather than $c_{t,l}(\mu)$. Therefore using~\eqref{basicid} it suffices to prove the following:

\medskip

\textbf{Claim 2:} \emph{Let $0 \le l \le n_0$. Define the quantities $d_{t,l}(\mu)$ via the identity \begin{equation}\label{basicid3}\begin{split}&\left(\sum_{t=-\infty}^\infty q^{(t+a(\mu\pi))(\frac12-s)} d_{t,l}(\mu)\right) \\&= \omega_\pi(-1)\left(\sum_{a=0}^\infty W_\pi(a(\varpi^a)) q^{-a(\frac12-s)}G(\varpi^{a-l}, \mu^{-1})\right) L(1-s, \mu^{-1}\omega_\pi^{-1}\pi)^{-1}.\end{split}\end{equation} If there exists $\mu\in \tilde{X}$ such that $a(\mu)\le l$ and $d_{t,l}(\mu) \neq 0$ then  $t \ge -\max(n, l +m)$. Further if $t = -\max(n, l +m)+r$ where $r \ge 0$ then $\sum_{\substack{\mu \in \tilde{X} \\ a(\mu)\le l}} \left|d_{t,l}(\mu) \right|^2  \ll q^{-r/2}.$
}

\medskip
We only consider the case $L(s, \pi)=1$, as the case $L(s, \pi) \neq 1$ is similar but easier. (Note that $L(s, \pi) \neq 1$ iff either $m=n$ or $n=1$.)

Let $\mu\in \tilde{X}$ be such that $a(\mu)\le l$. As $L(s, \pi)=1$, we can use the well-known formulas stated in~\cite[Equation (6) and Lemma 2.5]{sahasupwhittaker} to deduce that the quantity on the RHS of~\eqref{basicid3} lying inside the bracket is a constant of absolute value $\ll q^{-l/2}$ if $a(\mu)=l$ or $a(\mu)=0, \ l=1$; and is equal to 0 otherwise. Furthermore, there are at most 2 characters $\mu \in \tilde{X}$ with $a(\mu) \le n_0$ and $L(s, \mu^{-1}\omega_\pi^{-1}\pi)\neq 1$ (this can be checked, for example, using the classification written down in~\cite[Sec. 2.2]{sahasupwhittaker}).

We henceforth assume that $a(\mu)=l$ or $a(\mu)=0, \ l=1$; else there is nothing to prove as $d_{t,l}(\mu)=0$. Suppose first that $L(s, \mu^{-1}\omega_\pi^{-1}\pi)=1$. Then by equating coefficients on both sides of~\eqref{basicid3}, we see that $d_{t,l}(\mu) \neq 0 \Rightarrow t = -a(\mu \pi) \ge -\max(n, a(\mu)+m) \ge -\max(n, l+m),$ using Lemma \ref{lemmatwist}.
Furthermore if $t = - \max(n, l+m) + r$, then
$$\sum_{\substack{\mu \in \tilde{X} \\ a(\mu)\in \{l,0\} \\ L(s, \mu\pi)=1}} \left|d_{t,l}(\mu) \right|^2  \ll \sum_{\substack{\mu \in \tilde{X} \\ a(\mu)\in \{l,0\} \\ a(\mu\pi)=\max(n, l+m) -r}} q^{-l} \ll q^{-r/2},$$ again using Lemma \ref{lemmatwist}.

Suppose next that $L(s, \mu^{-1}\omega_\pi^{-1}\pi)\neq 1$. In this case $\mu \ne 1$ so if $d_{t,l}(\mu)\ne 0$ we must have $a(\mu)=l$. Also, the right side of~\eqref{basicid3} is of the form $\alpha_0 + \alpha_1 q^{-1(\frac12 -s)} + \alpha_2 q^{-2(\frac12 -s)}$ with $\alpha_i \ll q^{-(l+i)/2}$. Furthermore if $\alpha_2 \neq 0$ then $a(\mu \pi)=0 \le n-2$ and if $\alpha_2=0$ then $\alpha_1 \neq 0$ and $a(\mu \pi) \le \max(n_0, m) \le n-1$. So again equating coefficients and using Lemma \ref{lemmatwist}, we see that  $d_{t,l}(\mu) \neq 0 \Rightarrow t \ge  -n \ge  -\max(n, l+m),$ and
furthermore if $t = - \max(n, l+m) + r$, then
$$\sum_{\substack{\mu \in \tilde{X} \\ a(\mu)= l \\ L(s, \mu\pi)\neq 1}} \left|d_{t,l}(\mu) \right|^2  \ll \sum_{i=0}^2\sum_{\substack{\mu \in \tilde{X} \\ a(\mu)= l \\ a(\mu\pi)=\max(n, l+m) -r-i}} q^{-l-i} \ll q^{-r/2},$$ again using Lemma \ref{lemmatwist}.
 Putting everything together, the proof of Claim 2 is complete. \end{proof}
Next, for any $g \in G$, define $$n_0(g) = \min(l(g), \ n-l(g)), \qquad q(g) = \max(n_0, \ n_0(g)  -n_1 +m).$$ We note the useful bounds $0 \le n_0(g) \le n_0$ and $n_0 \le q(g) \le n_0 + m_1$.

\begin{proposition}\label{localwsupportfinal}Suppose that $g \in K a(\varpi^{n_1})$. Assume further that either $n$ is even or $l(g) \le n_0$. Then the following hold \begin{enumerate}
\item If  for some $y \in F^\times$, we have $W_\pi(a(y)g) \neq 0$, then  $v(y) \ge -q(g)$.

\item Suppose $b = -q(g)+r$ where $r \ge 0$. Then we have $$\left(\int_{v \in \OF^\times} \left|W_\pi( a(\varpi^bv)g) \right|^2 d^\times v\right)^{1/2} \ll q^{-r/4}.$$

\end{enumerate}
\end{proposition}
\begin{proof}This follows immediately by putting together Lemma~\ref{lemmatg} and Proposition~\ref{whitsuppsizeram}.
\end{proof}

\begin{remark}Note that the map on $\OF^\times$ given by $v \mapsto |W_\pi(a(vy)g)|$ is $U_{n_0(g)}$ invariant for all $y \in F^\times$, $g\in G$. Hence the second part of the above proposition is equivalent to $$\frac{1}{|\OF^\times/ U_{n_0(g)}|} \sum_{v \in \OF^\times/U_{n_0(g)}} \left|W_\pi( a(\varpi^bv)g) \right|^2 d^\times v \ll q^{-r/2}.$$
\end{remark}

\subsection{Test functions}\label{s:mcbeg}
We now change gears and start looking at certain local test functions (related to matrix coefficients) that will be used later by us in the trace formula. We begin with some definitions. Let $C_c^\infty(G, \omega_\pi^{-1})$ be the space of functions $\kappa$ on $G$  with the following properties:

 \begin{enumerate}

 \item $\kappa(z(y)g) = \omega_\pi^{-1}(y)\kappa(g)$.

 \item $\kappa$ is locally constant.

 \item   $|\kappa|$ is compactly supported on $Z\bs G$.

 \end{enumerate}
  Given $\kappa_1$, $\kappa_2$ in $C_c^\infty(G, \omega_\pi^{-1})$ we define the convolution $\kappa_1 \ast \kappa_2\in C_c^\infty(G, \omega_\pi^{-1})$ via \begin{equation}(\kappa_1 \ast \kappa_2) (h) = \int_{Z\bs G} \kappa_1(g^{-1})\kappa_2(gh) dg, \end{equation}
which turns $C_c^\infty(G, \omega_\pi^{-1})$ into an associative algebra.

Next let $\sigma$ be a representation of $G$ with central character equal to $\omega_\pi$. Then, for any $\kappa \in C_c^\infty(G, \omega_\pi^{-1})$, and any vector $v \in \sigma$, we define $R(\kappa) v$ to be the vector in $\sigma$ given by
\begin{equation}R(\kappa)v = \int_{Z\bs G} \kappa(g)(\sigma(g)v)  \ dg. \end{equation}

Let $v_\pi$ be any \emph{newform} in the space of $\pi$, i.e., any non-zero vector fixed by $K_1(\p^n)$. Equivalently $v_\pi$ can be any vector in $\pi$  corresponding to $W_\pi$ under some isomorphism $\pi \simeq \mathcal{W}(\pi,\psi)$. Thus $v_\pi$ is unique up to multiples. Put $v_\pi' = \pi(a(\varpi^{n_1}))v_\pi$. Note that $v_\pi'$ is up to multiples the unique non-zero vector in $\pi$ that is invariant under the subgroup $a(\varpi^{n_1})K_1(\p^n)a(\varpi^{-n_1})$.

Let $\langle, \rangle$ be any $G$-invariant inner product on $\pi$ (this is also unique up to multiples).  Define a matrix coefficient $\Phi_\pi$ on $G$ as follows:

$$\Phi_\pi(g) =
\frac{\langle
v_\pi , \pi(g) v_\pi\rangle}{\langle v_\pi , v_\pi\rangle}$$ which is clearly independent of the choice of $v_\pi$ or the normalization of inner product.

Let $$K^0 := K^0(\p^{n_1-n_0}) = \begin{cases} K &\text{ if } n \text{ is even, } \\ K^0(\p) &\text{ if } n \text{ is odd. }\end{cases}$$ Put $$\Phi_\pi'(g) = \begin{cases} \Phi_\pi(a(\varpi^{-n_1})ga(\varpi^{n_1})) = \frac{\langle
v_\pi' , \pi(g) v_\pi'\rangle}{\langle v_\pi' , v_\pi'\rangle} & \text{ if } g \in ZK^0, \\ 0 & \text{ if } g \notin ZK^0. \end{cases}$$

Then it follows that $\Phi_\pi' \in C_c^\infty(G, \omega_\pi^{-1})$ and $\Phi_\pi'(g^{-1}) = \overline{\Phi_\pi'(g)}$. In particular, the operator $R(\Phi'_\pi)$ is self-adjoint.

\begin{proposition}\label{keymatrixprop}There exists a positive real constant $\delta_\pi$ depending only on $\pi$ and satisfying $\delta_\pi \gg q^{-n_1-m_1}$ such that the following hold.
\begin{enumerate}

\item $R(\Phi'_\pi)v_\pi' = \delta_\pi v_\pi'$,

\item $\Phi'_\pi \ast \Phi'_\pi = \delta_\pi \Phi'_\pi.$

\end{enumerate}

\end{proposition}

\begin{remark}The above Proposition is a refinement of a result of Marshall~\cite{marsh15}, who proved a similar result in the special case  $\omega_\pi=1$, using a slightly different test function which does not differentiate between $n$ odd and even.
\end{remark}

\begin{remark} In fact with some additional work one can prove $\delta_\pi \asymp q^{-n_1 - m_1}.$
\end{remark}
The rest of this section will be devoted to proving this proposition. We note a useful corollary.

\begin{corollary}\label{nonneg}Let $\sigma$ be a generic
irreducible admissible unitarizable
representation of $G$ such that $\omega_\sigma = \omega_\pi$ and let $v_\sigma$ be any vector in the space of $\sigma$. Suppose that $R(\Phi'_\pi)v_\sigma = \delta v_\sigma$ for some complex number $\delta$. Then $\delta \in \{0, \delta_\pi\};$ in particular, $\delta$ is a non-negative real number.
\end{corollary}
\begin{proof}We have $$\delta \delta_\pi v_\sigma = \delta_\pi R(\Phi'_\pi)v_\sigma = R(\Phi'_\pi \ast \Phi'_\pi)v_\sigma = R(\Phi'_\pi) R(\Phi'_\pi)v_\sigma=\delta^2 v_\sigma,$$ implying that $\delta \in \{0, \delta_\pi\}.$
\end{proof}

\subsection{Some preparatory lemmas}

\begin{lemma}\label{preplemma1}Consider the representation $\pi|_{K^0}$ of $K^0$ and let $\pi'$ be the subrepresentation of $\pi|_{K^0}$ generated by $v_\pi'$. Then $\pi'$ is a finite dimensional irreducible representation of $K^0$.
\end{lemma}
\begin{proof}We know that $\pi'$ is isomorphic to a direct sum of irreducible
representations of $K^0$. However if there were more than one summand in the
decomposition of $\pi'$, then the representation $\pi|_{K^0}$  (and hence the representation $\pi$) would contain a
$a(\varpi^{n_1})K_1(\p^n)a(\varpi^{-n_1})$-fixed subspace of dimension greater than one; by
newform theory this is impossible. Hence $\pi'$ is irreducible. The finite dimensionality of $\pi'$ follows from the admissibility of $\pi$.
\end{proof}

\begin{lemma}Let $\pi'$ be as in the above Lemma. Then both the claims of Proposition~\ref{keymatrixprop} hold with the quantity $\delta_\pi$ defined as follows: $$\delta_\pi = \ \int_{Z\bs G} |\Phi_\pi'(g)|^2 dg = \ \int_{K^0} |\Phi_\pi'(g)|^2 dg = \  \frac{1}{[K:K^0] \ \mathrm{dim}(\pi')}.$$
\end{lemma}
\begin{proof}
Note that $\langle , \rangle$ is an invariant inner product for $\pi'$. It follows immediately (from the orthonormality of matrix coefficients) that the last two quantities are equal. The equality of the middle two quantities is immediate from our normalization of Haar measures.

We now show that this quantity satisfies the claims of Proposition~\ref{keymatrixprop}. First of all,  $R(\Phi_\pi')v_\pi'$ is a vector in $\pi$ that is invariant under the subgroup $a(\varpi^{n_1})K_1(\p^n)a(\varpi^{-n_1})$. It follows that $R(\Phi_\pi')v_\pi' = \delta v_\pi'$ for some constant $\delta.$ Taking inner products with $v_\pi'$ immediately shows that $\delta = \delta_\pi$. This proves the first assertion of the Proposition. The second assertion is a standard property of convolutions of matrix coefficients.
\end{proof}

\subsubsection*{Proof of Proposition~\ref{keymatrixprop} in the case of non-supercuspidal representations}
We now prove Proposition~\ref{keymatrixprop} for all non-supercuspidal representations $\pi$. It suffices to show that $$\dim(\pi') \ll q^{n_0+m_1},$$ where $\pi'$ is as in Lemma~\ref{preplemma1}.

We can embed $\pi$ inside a representation $\chi_1 \boxplus \chi_2$, consisting of
smooth functions $f$ on $G$ satisfying $$f\left(\mat{a}{b}{0}{d} g\right) =
|a/d|^{\frac12} \chi_1(a) \chi_2(d)  f(g).$$ Here $\chi_1$ and $\chi_2$ are two (not necessarily unitary) characters. Let $f'$ be the function in
$\chi_1 \boxplus \chi_2$ that corresponds to $v_\pi'$. Let $K'$ be the (normal) subgroup of $K^0$ consisting of matrices $\mat{a}{b}{c}{d}$ such that $a \equiv d \equiv 1 \pmod{\p^{n_0+m_1}}$, $b \equiv 0 \pmod{\p^{n_1+m_1}}$, $c \equiv 0 \pmod{\p^{n_0+m_1}}$. Let $V_{K'}$ be the
subspace of $\chi_1 \boxplus \chi_2$ consisting of the
functions $f$ that satisfy $f(gk) =  \omega_\pi(a) f(g)$ for all $k=\mat{a}{b}{c}{d} \in K'$. Then $f' \in V_{K'}$. Moreover (and this is the key fact!) if $k \in K^0$ and $k' \in K'$, then the top left entries of $k'$ and $kk'k^{-1}$ (both these matrices are elements of $K'$) are equal modulo $\p^{m}$. Hence the space $V_{K'}$ is  stable under the action of $K^0$. So it suffices to prove that
$\mathrm{dim}(V_{K'}) \ll q^{n_0+m_1}$.

Using the Iwasawa decomposition, it follows easily that $|B(F)\bs
G(F)/K'| \asymp q^{n_0+m_1}.$ Fix a set of double coset representatives $S$ for $B(F)\bs
G(F)/K'$. Since any element of $V_{K'}$ is uniquely
determined by its values on $S$, it follows that  $\mathrm{dim}(V_{K'}) \ll q^{n_0+m_1}$. The proof
is complete.

\subsection{Proof of Proposition~\ref{keymatrixprop} in the case of supercuspidal representations} \label{s:endmc}

We now assume that $\pi$ is supercuspidal. In this case, $m \le n_0$, hence $m_1=0$. So, we suffices to prove that \begin{equation} \label{reqprove}\int_{K^0} |\Phi_\pi(a(\varpi^{-n_1})ga(\varpi^{n_1}))|^2 dg \gg q^{-n_1}.\end{equation}

The next Proposition gives a formula for $\Phi_\pi$, which may be of independent interest.

\begin{proposition}\label{mcsupercusp} For $0 \le l<n$, we have
\begin{equation}\begin{split}\Phi_\pi(n(x)&g_{t,l,v})  =
G(-\varpi^{l-n}, 1)G(\varpi^{t+l}v^{-1}-x, 1 )\omega_\pi(-v) \delta_{t,-2l}
\\&+ \eps(\frac12, \pi) \omega_\pi(v)\sum_{\substack{\mu \in
\tilde{X} \\ a(\mu) =n-l\\  \ a(\mu \tilde{\pi}) = n-2l-t}}
G(\varpi^{l-n}, \mu)G(vx - \varpi^{t+l},  \mu)\varepsilon(\frac12,
\mu \tilde{\pi}) .\end{split}\end{equation}

\end{proposition}
\begin{proof}
Using the usual inner product in the Whittaker model, and the fact that $W_\pi(a(t))$ is supported on $t \in \OF^\times$, (as $\pi$ is supercuspidal) it follows that \begin{equation}\label{phiform1}\Phi_\pi(n(x) g_{t,l,v})= \int_{\OF^\times} \psi(-ux) \omega_\pi(u) \overline{W_\pi(g_{t,l,vu^{-1}})} d^\times u.\end{equation}

On the other hand, by the formula~\cite[Prop. 2.30]{sahasupwhittaker} for $W_\pi$, and using Proposition \ref{atkin-lehnerwhit} we have

\begin{align*}W_{\pi}(a(\varpi^t)&wn(\varpi^{-l}v)) = \omega_\pi(-v^{-1})\psi(-\varpi^{t+l}v^{-1})  G(\varpi^{l-n}, 1) \delta_{t,-2l} \\&  +  \eps(1/2,  \tilde{\pi}) \omega_\pi(v^{-1})\psi(-\varpi^{t+l}v^{-1}) \  \sum_{\substack{\mu \in \tilde{X} \\ a(\mu) =n-l\\  \ a(\mu \pi) = n-t-2l}} G(\varpi^{l-n}, \mu^{-1}) \ \varepsilon(1/2, \mu^{-1} \pi) \mu(-v).
\end{align*}

Substituting this into~\eqref{phiform1}, we immediately get the required result.
\end{proof}

To obtain~\eqref{reqprove}, we will need to substitute the formula from the above proposition and integrate. The following elementary lemma (which is similar to Lemma 2.6 of \cite{yueke}) will be useful; we omit its proof.

\begin{lemma}\label{integlemma}Let $f$ be a function on $G$ that is right $K_1(\p^n)$-invariant.
Then $$\int_{ G} f(g) dg = \sum_{k=0}^{n}A_k \int_{ B} f(bwn(\varpi^{-k}))db,$$
where $A_0 = (1+q^{-1})^{-1}$, $A_n = q^{n}(1+q^{-1})^{-1}$, and for $0<k<n$,
$A_k=q^{k}(1-q^{-1})(1+q^{-1})^{-1}.$
\end{lemma}

We now complete the proof of \eqref{reqprove}. Using Lemma \ref{integlemma}, it suffices to prove that \begin{equation} \label{reqprove2}\int_{\substack{b\in B\\bwn(\varpi^{-n_1}) \in a(\varpi^{-n_1})K^0a(\varpi^{n_1})}} |\Phi_\pi(bwn(\varpi^{-n_1}))|^2 db \gg q^{-2n_1}.\end{equation}

Now, note that the quantity $z(u)n(x)a(y)wn(\varpi^{-n_1})$ lies in
$a(\varpi^{-n_1})K^0a(\varpi^{n_1})$ if and only if :

$$u = \varpi^{n_1}u', \ y =\varpi^{-2n_1}y', \ x = \varpi^{-n_1} x', \ y' \in \OF^\times, u' \in \OF^\times, \ x'\in \OF, \ y' - x' \in \p^{n_1-n_0}.$$

Hence the left side of \eqref{reqprove2} is equal to \begin{equation}\label{reqprove3}q^{-n_1}\int_{\substack{y' \in \OF^\times, \ x' \in \OF \\ x'\in y' + \p^{n_1-n_0}}} |\Phi_\pi(n(\varpi^{-n_1} x')g_{-2n_1, -n_1, y'^{-1}})|^2 dx' d^\times y'  \end{equation}

Now, we can exactly evaluate the integral in \eqref{reqprove3} using    Proposition \ref{mcsupercusp}. We expand out $|\Phi_\pi(n(\varpi^{-n_1} x')g_{-2n_1, n_1, y'^{-1}})|^2$ and observe that the main (diagonal) terms are simple to evaluate as we know the modulus-squared of Gauss sums. Indeed, the contribution to \eqref{reqprove3} from the diagonal terms is simply $$ q^{-n_1}\int_{\substack{y' \in \OF^\times, \ x' \in \OF \\ x'\in y' + \p^{n_1-n_0}}} \sum_{\substack{\mu \in \tilde{X} \\ a(\mu)=n_0 \\ a(\mu \tilde{\pi}) = n}} q^{-2n_0} \asymp q^{-2n_1}.$$ On the other hand, the contribution from the cross terms is zero. Indeed, each cross term involves an integral like $\int_{\substack{y' \in \OF^\times, \ x' \in \OF \\  y'^{-1}x'-1 \in \p^{n_1-n_0}\OF^\times}} \mu_1^{-1}\mu_2(( y'^{-1}x'-1 ) \varpi^{n_0-n_1})$ which equals 0 because of the orthogonality of characters. This completes the proof of \eqref{reqprove2}.

\section{Supnorms of global newforms}\label{s:global}
From now on, we move to a global setup and consider newforms on $\GL_2(\A)$ where $\A$ is the ring of adeles over $\Q$. For any place $v$ of $\Q$, we will use the notation $X_v$ for each \emph{local object} $X$ introduced in the previous section. The corresponding global objects will be typically denoted without the subscript $v$. The archimedean place will be denoted by $v=\infty$. We will usually denote a non-archimedean place $v$ by $p$ where $p$ is a rational prime. The set of all non-archimedean places (primes) will be denoted by $\f$.

We fix measures on all our adelic groups (like $\A$, $\GL_2(\A)$, etc.) by taking the product of the local measures
over all places (for the non-archimedean places, these local measures were normalized in Section~\ref{sec:2-notations}; at the archimedean place we fix once and for all a suitable Haar measure). We normalize the Haar measure on $\R$ to be the usual Lebesgue measure. We give all discrete groups the
counting measure and thus obtain a measure on the appropriate quotient groups.

\subsection{Statement of result}\label{s:globalstatement} As usual, let $G=\GL_2$. Let $\pi = \otimes_v \pi_v$ be an irreducible, unitary, cuspidal automorphic
representation of $G(\A)$ with central character $\omega_\pi = \prod_v
\omega_{\pi_v}$. For each prime $p$, let the integers $n_p$, $n_{1,p}$, $n_{0,p}$, $m_p$, $m_{1,p}$ be defined as in Section~\ref{s:repnew}. We put $N=\prod_p p^{n_p}$, $N_0=\prod_p p^{n_{0,p}}$, $N_1=\prod_p p^{n_{1,p}},$  $M = \prod_p p^{m_p}$, $M_1 = \prod_p p^{m_{1,p}}$.
 Thus, $N$ is the conductor of $\pi$, $M$ is the conductor of $\omega_\pi$, $N_0$ is the largest integer such that $N_0^2 |N$, and $N_1=N/N_0$ is the smallest integer such that $N|N_1^2$. Let $N_2=N_1/N_0 = N/N_0^2$. Note that $N_2$ is a squarefree integer and is the product of all the primes $p$ such that $p$ divides $N$ to an odd power. If $N$ is squarefree, then $N_2=N_1=N$ and $N_0 = 1$ while if $N$ is a perfect square then $N_0 = N_1 = \sqrt{N}$ and $N_2=1$. Note also that $M_1 = M /\gcd(M, N_1)$.

We assume that $\pi_\infty$ is a spherical principal series representation whose central character is trivial on $\R^+$. This means that $\pi_\infty \simeq \chi_1 \boxplus \chi_2,$\footnote{For two characters $\chi_1$, $\chi_2$ on $\R^\times$, we let $\chi_1
\boxplus \chi_2$ denote the principal series representation on
$G(\R)$
that is unitarily induced from the corresponding representation of  $B(\R)$; this consists of smooth functions $f$ on $G(\R)$ satisfying $$f\left(\mat{a}{b}{0}{d} g\right) = |a/d|^{\frac12} \chi_1(a) \chi_2(d)  f(g).$$} where for $i=1, 2$, we have $\chi_1 = |y|^{it} \sgn(y)^{m}$, $\chi_2 = |y|^{-it} \sgn(y)^{m}$, with $m \in \{0,1\}$,  $t \in \R \cup (-\frac{i}{2},\frac{i}{2})$.

Let $K_1(N) = \prod_{p \in \f} K_{1,p}(p^{n_p}) = \prod_{p \nmid N} G(\Z_p) \prod_{p|N}  K_{1,p}(p^{n_p})$ be the standard congruence subgroup of $G(\hat{\Z}) = \prod_{p \in \f}G(\Z_p)$; note that $K_1(N)G(\R)^+ \cap G(\Q)$ is equal to the standard congruence subgroup $\Gamma_1(N)$ of $\SL_2(\Z)$. Let $K_\infty = \SO_2(\R)$ be the maximal connected compact subgroup of $G(\R)$ (equivalently, the maximal compact subgroup of $G(\R)^+$). We say that a non-zero automorphic form $\phi \in V_\pi$ is a \emph{newform} if $\phi$
is $K_1(N)K_{\infty}$-invariant.
It is well-known that a newform $\phi$ exists and is unique up to multiples, and
corresponds to a factorizable vector $\phi= \otimes_v \phi_v$. We define
$\|\phi\|_2 = \int_{Z(\A)G(F)\bs G(\A)} |\phi(g)|^2 dg.$

\begin{remark}\label{r:adelic}If $\phi$ is a newform, then the function $f$ on $\H$ defined by $f(g(i)) = \phi(g)$ for each $g \in \SL_2(\R)$ is a Hecke-Maass cuspidal newform of level $N$ (and character $\omega_\pi$). Precisely, it satisfies the relation \begin{equation}\label{masform} f\left(\mat{a}{b}{c}{d} z\right) = \left(\prod_{p|N}\omega_{\pi, p}(d) \right)f(z) \quad \text{ for all }\mat{a}{b}{c}{d} \in \Gamma_0(N).\end{equation}  The Laplace eigenvalue $\lambda$ for $f$ is given by $\lambda = \frac{1}{4} + t^2$ where $t$ is as above. (Note that $\lambda \asymp (1+|t|)^2.$)

Furthermore, any Hecke-Maass cuspidal newform $f$ is obtained in the above manner from a newform $\phi$ in a suitable automorphic representation $\pi$. The newform $\phi$ can be directly constructed from $f$ via strong approximation. It is clear that $\sup_{g \in G(\A)} |\phi(g) | = \sup_{z\in \Gamma_0(N) \bs \H} |f(z)|$.
\end{remark}

Our main result is as follows.

\begin{theorem}\label{t:globalmain}Let $\pi$ be an irreducible, unitary, cuspidal automorphic
representation of $G(\A)$ such that $\pi_\infty \simeq \chi_1 \boxplus \chi_2$, where for $i=1, 2$, we have $\chi_1 = |y|^{it} \sgn(y)^{m}$, $\chi_2 = |y|^{-it} \sgn(y)^{m}$, with $m \in \{0,1\}$, $t \in \R \cup (-\frac{i}{2},\frac{i}{2})$. Let the integers $N_0$, $N_1$, $M_1$ be defined as above and let $\phi \in V_\pi$ be a newform satisfying $\|\phi\|_2 = 1$. Then $$\sup_{g \in G(\A)} |\phi(g)| \ll_\eps  N_0^{1/6 + \eps} N_1^{1/3+\eps} M_1^{1/2} \ (1 +|t|)^{5/12+\eps}.$$

\end{theorem}

\begin{remark} Assume that $\pi$ has trivial central character and $N_1 \asymp \sqrt{N}$ (this is the case whenever $N$ is sufficiently ``powerful"). Then we  get $\sup_{g \in G(\A)} |\phi(g)| \ll_{t, \eps} N^{\frac{1}{4}+\eps} ,$ which is a considerable improvement of the best previously known result $\sup_{g \in G(\A)} |\phi(g)| \ll_{t,\eps} N^{\frac{5}{12}+\eps}$ due to the author~\cite{sahasuplevel}.
\end{remark}

\subsection{Atkin-Lehner operators and a generating domain}\label{s:strategy}
Let $\pi$ be as in Section \ref{s:globalstatement} and $\phi \in V_\pi$ a newform. In order to prove Theorem \ref{t:globalmain} we will restrict the variable $g$ to a carefully chosen generating domain inside $G(\A)$. In order to do this, we will have to consider the newform $\phi$ along with some of its Atkin-Lehner translates. The object of this section is to explain these ideas and  describe our generating domain. The main result in this context is Proposition \ref{p:funddom} below.

We begin with some definitions.
For any integer $L$, let $\P(L)$ denote the set of distinct primes dividing $L$. For any subset $S$ of $\P(N)$, let $\eta_S, h_S \in G(\A_\f)$ be defined as follows: $\eta_{S,p} = \mat{}{1}{p^{n_p}}{}$ if $p\in S$, $\eta_{S,p} =1$ otherwise; $h_{S,p} = a(p^{n_{1,p}})$ if $p\in S$, $h_{S,p} =1$ otherwise. Define  $$K_S = \prod_{p \in S}G(\Z_p) \subset G(\A_\f), \quad J_S = K_Sh_S \subset  G(\A_\f).$$ Finally, define $$\J_S = \{g \in J_S: l(g_p)\le n_{0,p} \text{ for all } p \in S \text{ such that } n_p \text{ is odd}\}.$$ Using Lemma \ref{usefull}, we see that $g \in \prod_{p \in S}G(\Q_p)$ belongs to $\J_S$ iff $g_p \in wK_p^0(p)a(p^{n_{1,p}})$ for all $p \in S$ for which $n_{p}$ is odd. If $L$ divides $N$, we abuse notation by denoting $$h_L = h_{\P(L)}, \quad K_L = K_{\P(L)}, \quad J_L = J_{\P(L)}, \quad \J_L= \J_{\P(L)}.$$ For any $0<c < \infty$, let $D_{c}$ be the subset of $B_1(\R)^+ \simeq \H$ defined by $D_{c}:= \{n(x)a(y): x\in \R, \ y\ge c\}.$ Finally, for $L>0$, define $$\F_{L} = \{n(x)a(y) \in D_{\sqrt{3}/(2L)}: z = x+iy \text{ satisfies } |cz+d|^2 \ge 1/L \quad \forall \ (0,0) \neq (c,d) \in \Z^2\}.$$

Next, for any subset $S$ of $\P(N)$, let $\omega_{\pi}^{S} = \prod_v \omega_{\pi, v}^{S} $ be the unique character\footnote{The existence, as well as uniqueness, of the character $\omega_{\pi}^{S}$ follows from the identity $\A^\times = \Q^\times \R^+ \prod_p \Z_p^\times$.} on $\Q^\times \bs \A^\times$ with the following properties:

\begin{enumerate}
\item $\omega_{\pi, \infty}^{S}$ is trivial on $\R^+$.
\item $\omega_{\pi,p}^{S}|_{\Z_p^\times}$ is trivial if $p \in S$ and equals $\omega_{\pi,p}|_{\Z_p^\times}$ if $p \notin S$.
\end{enumerate}

Note that $\omega_{\pi}^{\P(N)} = 1$, $\omega_{\pi}^{\emptyset} = \omega_\pi$, and for each $S$, $\omega_{\pi}^{S}$ has conductor $\prod_{p \notin S}p^{m_p}$. Define the irreducible, unitary, cuspidal, automorphic representation $\pi^{S}$ by $\pi^{S} = \tilde{\pi} \otimes \omega_\pi^{S} = \pi \otimes (\omega_\pi^{-1} \omega_\pi^S).$  A key observation is that for every $S$, the representation $\pi^{S}$ has conductor $N$ and its central character $\omega_{\pi^{S}} = \omega_\pi^{-1} (\omega_{\pi}^{S})^2 $ has conductor $M$. We have $\pi^{\emptyset} = \pi$ and $\pi^{\P(N)} = \tilde{\pi}.$

\begin{lemma}The function $\phi^S$ on $G(\A)$ given by $\phi^S(g) := (\omega_\pi^{-1} \omega_\pi^S)(\det(g))\phi(g \eta_S)$ is a newform in $\pi^{S}$.

\end{lemma}
\begin{proof}It is clear that $\phi^S$ is a vector in $\pi^{S}$, and one can easily check from the defining relation that it is $K_1(N)K_{\infty}$ invariant.
\end{proof}
\begin{remark} In the special case $\omega_\pi = 1$, one has $\pi^S = \pi$ for every subset $S$ of $\P(N)$. In this case, for each $S$, the involution $\pi(\eta_S)$ on  $V_\pi$ corresponds to a classical Atkin-Lehner operator, and $\phi^S = \pm \phi$ with the sign equal to the Atkin-Lehner eigenvalue. We will call the natural map on $Z(\A) G(\Q) \bs G(\A) /K_1(N)K_\infty$ induced by $g \mapsto g \eta_S$ the adelic Atkin-Lehner operator associated to $S$.
\end{remark}

Recall that $\J_{N} = \prod_{p|N_2}wK_p^0(p) a(p^{n_{1,p}}) \prod_{p|N, \ p\nmid N_2} G(\Z_p) a(p^{n_{1,p}}) \subset G(\A_\f)$. The next Proposition tells us that any point in $Z(\A) G(\Q) \bs G(\A) /K_1(N)K_\infty$ can be moved by an adelic Atkin-Lehner operator to a point whose finite part lies in $\J_N $ and whose infinite component lies in $\F_{N_2}$.

\begin{proposition}\label{p:funddom} Suppose that $g \in G(\A).$ Then there exists a subset $S$ of $\P(N_2)$ such that $$g \in Z(\A) G(\Q) \left(\J_N  \times \F_{N_2}\right) \eta_S K_1(N)K_\infty.$$
\end{proposition}
\begin{proof}  Let $w_N$ be the diagonal embedding of $w=\mat{0}{1}{-1}{0}$ into $K_N$. The determinant map from $w_N h_{N}K_1(N)h_{N}^{-1}w_N^{-1}$ is surjective onto $\prod_p \Z_p^\times$. Hence by strong approximation for $gh_{N}^{-1}w_N^{-1}$, we can write $gh_{N}^{-1}w_N^{-1} = z g_\Q g_\infty^+ (w_N h_{N}k h_{N}^{-1}w_N^{-1})$ where $z \in Z(\A)$, $g_\Q \in G(\Q)$, $g_\infty^+ \in G(\R)^+$, $k \in K_1(N)$. In other words, \begin{equation}\label{e:aleq1}g \in Z(\A)G(\Q) g_\infty^+ w_N h_{N}K_1(N).\end{equation} Using Lemma 1 from \cite{harcos-templier-1}, we can find a divisor $N_3$ of $N_2$, and a matrix $W \in M_2(\Z)$ such that $$W \equiv \begin{bmatrix}0&*\\0&0\end{bmatrix} \bmod{N_3}, \ W \equiv \begin{bmatrix}*&*\\0&*\end{bmatrix} \bmod{N_2}, \ \det(W)=N_3,  \ W_\infty g_\infty^+ \in \F_{N_2}K_\infty.$$ Above, $W_\infty$ denotes the element $W$ considered as an element of $G(\R)^+$. Let $S$ be the set of primes dividing $N_3$. Note that $W_p \in K_{0,p}(p)\mat{0}{1}{p}{0}$ if $p \in S$, $W_p \in K_{0,p}(p)$ if $p|N_2$ but $p \notin S$, and $W_p \in G(\Z_p)$ if $p \nmid N_2$. Since $W \in G(\Q)$, it follows from the above and from \eqref{e:aleq1} that \begin{align*}g &\in Z(\A)G(\Q) \F_{N_2}K_\infty  \left(\prod_{p \in  S} K_{0,p}(p)\mat{0}{1}{p}{0}  w \right)\left(\prod_{\substack{p|N_2 \\ p\notin S}} K_{0,p}(p)w\right) \left(\prod_{\substack{p|N \\ p \nmid N_2}} G(\Z_p)\right) h_{N}K_1(N) \\ &= Z(\A) G(\Q) \left(\J_N  \times \F_{N_2}\right) \eta_S K_1(N)K_\infty,\end{align*} where in the last step we have used Lemma \ref{lemma:ALnew}.
\end{proof}

\begin{corollary}\label{cor:supnorm}Let $\pi$, $\phi$ be as in Theorem \ref{t:globalmain}. Suppose that for all subsets $S$ of $\P(N_2)$ and all $g \in \J_N $, $n(x)a(y) \in \F_{N_2}$, we have $$|\phi^S(g n(x)a(y))| \ll_\eps N_1^{1/2+\eps} M_1^{1/2} N_2^{-1/6} \ (1 +|t|)^{5/12+\eps}.$$ Then the conclusion of Theorem \ref{t:globalmain} is true.
\end{corollary}
\begin{proof}This follows from the above Proposition and the fact $|\phi^S(g n(x)a(y))| = |\phi(g n(x)a(y) \eta_S)|.$
\end{proof}

\subsection{Sketch of proof modulo technicalities}\label{s:proofsketch}
In this subsection, we prove Theorem \ref{t:globalmain} assuming some key bounds whose proofs will take the rest of this paper. For brevity we put $T=1+|t|$. Also, recall that $N_2=N_1/N_0$. We need to show that for each $g \in G(\A)$, $$|\phi(g)| \ll_\eps N_1^{1/2+\eps} M_1^{1/2} N_2^{-1/6} \ T^{5/12+\eps}.$$ By letting $\phi$ run over all its various Atkin-Lehner translates $\phi^S$, $S \subseteq \P(N_2)$,  we may assume (by Corollary \ref{cor:supnorm}) that $g  \in \J_N \F_{N_2}$. Therefore in what follows, we will not explicitly keep track of the set $S$, but instead prove the following: \emph{Given an  automorphic representation $\pi$ as in Section \ref{s:globalstatement} (with associated quantities $N_1$, $N_2$, $T$, $M_1$ as defined earlier), a newform $\phi \in V_\pi$ satisfying $\|\phi\|_2 =1$, and elements $g \in \J_N $, $n(x)a(y) \in \F_{N_2}$, we have}
\begin{equation}\label{e:reqbdmain}|\phi(g n(x)a(y))| \ll_\eps N_1^{1/2+\eps} M_1^{1/2} N_2^{-1/6} \ T^{5/12+\eps}.\end{equation}

As noted, the above statement implies Theorem \ref{t:globalmain}. Implicit here is the fact that  we are letting $\pi$ vary among the various $\pi^S$, which all have exactly the same values of $N, N_1, N_2, M_1, T$ as $\pi$ does, and moreover the corresponding newforms $\phi^S$ all satisfy $\|\phi^S\|_2 = \|\phi\|_2$.

We will prove \eqref{e:reqbdmain} by a combination of two methods. First, in Proposition \ref{fourierboundprop}, we will use the Whittaker expansion to bound this quantity. Precisely, we will prove the following bound:
\begin{equation}\label{finalfourierbd}|\phi(g n(x)a(y))| \ll_\eps  (NT)^\eps \left( \left(\frac{N_1 M_1 T}{N_2y} \right)^{1/2} + \left(\frac{N_1 T^{1/3}}{N_2} \right)^{1/2} \right).   \end{equation}
To prove the above bound, we will rely on Proposition \ref{localwsupportfinal}. Next, in Proposition \ref{prop:ampl}, we will use the amplification method to bound this quantity.  We will prove that for each $\Lambda \ge 1$, we have
\begin{equation}\label{finalamplbd}|\phi(g n(x)a(y))|^2 \ll_\eps (NT\Lambda)^\eps \ N_1 M_1 \left[\frac{T + N_2^{1/2}T^{1/2}y}{\Lambda}   + \Lambda^{1/2}T^{1/2}(N_2^{-1/2} + y)  + \Lambda^2 T^{1/2}N_2^{-1}  \right].   \end{equation}
The proof of this bound will rely on Proposition \ref{keymatrixprop} and some counting arguments due to Harcos and Templier. Combining the two bounds will lead to Theorem~\ref{t:globalmain}, as we explain now.

Choose $\Lambda = T^{1/6}N_2^{1/3}$. Then~\eqref{finalamplbd} becomes \begin{equation}\begin{split}\label{finalamplbd2}
|\phi(g n(x)a(y))|^2 \ll_\eps (NT)^\eps \ N_1 M_1 &\bigg[T^{5/6}N_2^{-1/3}   + T^{7/12}N_2^{-1/6}y   \bigg]\end{split}.   \end{equation}

If $y \le  T^{1/4}N_2^{-1/6}$, then we use \eqref{finalamplbd2} to immediately deduce~\eqref{e:reqbdmain}. If $y \ge T^{1/4}N_2^{-1/6}$, then we use~\eqref{finalfourierbd} to obtain the  bound
\begin{equation}\label{finalfourierbd2}|\phi(g n(x)a(y))| \ll_\eps (NT)^\eps M_1^{1/2}N_1^{1/2}N_2^{-5/12} T^{3/8}  \end{equation} which is much stronger than~\eqref{e:reqbdmain}! This completes the proof.

\subsection{The bound via the Whittaker expansion}\label{s:fourierglobal}
Let $\pi$, $\phi$ be as in Section \ref{s:globalstatement} with $\| \phi \|_2=1$. The object of this section is to prove the following result.

\begin{proposition}\label{fourierboundprop}Let $x \in \R$, $y \in \R^+$, $g \in \J_N $. Then $$|\phi(g n(x)a(y))| \ll_\eps (NT)^\eps \left( \left(\frac{N_0 M_1 T}{y} \right)^{1/2} + \left(N_0 T^{1/3} \right)^{1/2} \right).                     $$

\end{proposition}

 \begin{remark}If we assume Conjecture~\ref{locallindelof} stated earlier, then we can improve the bound in Proposition~\ref{fourierboundprop} to $(NT)^\eps \left( \left(\frac{N_0 M_1 T}{y} \right)^{1/2} + \left( T^{1/3} \right)^{1/2} \right).$
 \end{remark}

 We now  begin the proof of Proposition~\ref{fourierboundprop}. One has the usual Fourier expansion at infinity \begin{equation}\label{fourier}\phi(n(x)a(y)) = y^{1/2} \sum_{n \in \Z_{\ne 0}} \rho_{\phi}(n)K_{it}(2 \pi |n| y) e(nx).\end{equation}

The next Lemma notes some key properties about the Fourier coefficients appearing in the above expansion.

\begin{lemma}\label{rholambdalemma}The Fourier coefficients $\rho_{\phi}(n)$ satisfy the following properties.

\begin{enumerate}
\item $|\rho_{\phi}(n)| =|\rho_{\phi}(1)\lambda_\pi(n)|$ where $\lambda_\pi(n)$ are the coefficients of the $L$-function of $\pi$.

\item   $|\rho_{\phi}(1)| \ll_\eps   (NT)^\eps e^{\pi t/2}.$

\item $\sum_{1 \le |n| \le X} |\lambda_\pi(n)|^2 \ll X (NTX)^\eps$.

\end{enumerate}

\end{lemma}
\begin{proof}All the parts are standard. The first part is a basic well-known relation between the Fourier coefficients and Hecke eigenvalues. The second part is due to Hoffstein-Lockhart~\cite{HL94}. The last part follows from the analytic properties of the Rankin-Selberg $L$-function (e.g., see~\cite{harcos-michel}).
\end{proof}

The Fourier expansion~\eqref{fourier} is a special case of the more general Whittaker expansion that we describe now. Let $g_\f \in G(\A_\f)$. Then the Whittaker expansion for $\phi$ says that \begin{equation}\label{whittakerexp}
\phi(g_\f n(x)a(y) )=
\sum_{q \in \mathbb{Q}_{\neq 0}}
W_\phi(a(q) g_\f n(x)a(y))\end{equation} where $W_\phi$ is a global Whittaker newform
corresponding to $\phi$ given explicitly by
$$W_\phi(g) = \int_{x \in \mathbb{A} / \mathbb{Q}} \phi(n(x) g)
\psi(-x) \, d x.$$ Putting $g_\f =1$ in \eqref{whittakerexp} gives us the expansion~\eqref{fourier}. On the other hand, the function $W_\phi$ factors as $W_\phi(g) = c \prod_v W_v(g_v)$ where
\begin{enumerate}
\item $W_p =  W_{\pi_p}$ at all finite primes $p$.

\item $|W_\infty(a(q)n(x)a(y))| = |qy|^{1/2} |K_{it}(2 \pi |q| y)|$

\end{enumerate}
The constant $c$ is related to $L(1, \pi, \mathrm{Ad})$; for further details on this constant, see \cite[Sec. 3.4]{sahasupwhittaker}.

For any $g = \prod_{p|N} g_p \in \J_N $, define

$$N_0^g = \prod_{p|N}p^{n_0(g_p)}, \quad Q^g = \prod_{p|N}p^{q(g_p)}$$
where the integers $n_0(g_p)$, $q(g_p)$ are as defined just before Proposition \ref{localwsupportfinal}. Note that the ``useful bounds" stated there imply that $N_0^g | N_0$ and $Q^g |N_0M_1$.

\begin{lemma}Suppose that  $g \in \J_N $ and $W_\phi(a(q) g n(x)a(y)) \neq 0$ for some $q \in \Q$. Then we have $q = \frac{n}{Q^g}$ for some $n \in \Z$.
\end{lemma}

\begin{proof} We have $W_{\pi_p}(a(q)g_p) \neq 0$ for each $p|N$ and $W_{\pi_p}(a(q)) \neq 0$ for each $p \nmid N$.  Now the result follows from Proposition~\ref{localwsupportfinal} and Lemma \ref{whitsuppdiagtriv}.
\end{proof}

Henceforth we fix some $g \in \J_N $. By comparing the expansion~\eqref{whittakerexp} for $g_\f = g$ with the trivial case $g_\f =1$, we conclude that
\begin{align}
\nonumber &\phi(g n(x)a(y) )= \sum_{n \in \Z_{\ne 0}}
W_\phi(a(n/Q^g) g n(x)a(y))\\ \nonumber&=\sum_{n \in \Z_{\ne 0}}
\left(\prod_{p|N} W_{\pi_p}(a(n/Q^g) g)\right) \left( c \prod_{p\nmid N}W_{\pi_p}(a(n) )\right) W_\infty(a(n/Q^g)n(x)a(y))\\
&\label{fourierexp2}=\left(\frac{y}{Q^g}\right)
^{1/2}\sum_{n \in \Z_{\ne 0}}  (|n|, N ^\infty)^{1/2}\rho_{\phi}\left(\frac{n}{(|n|, N ^\infty)}\right)\lambda_{\pi_{N}}(n; g)K_{it}\left(\frac{2 \pi |n|y }{Q^g}\right) \chi_n
\end{align}

where $\chi_n$  is some complex number of absolute value 1, and for each non-negative integer $n$ we define

$$\lambda_{\pi_{N}}(n; g) := \prod_{p|N}W_{\pi_p}\left(a(n p^{-q(g_p)})g_p\right).$$

The tail of the sum \eqref{fourierexp2} consisting of the terms with $2 \pi |n|y/Q^g > T + T^{1/3 + \eps}$ is negligible because of the exponential decay of the Bessel function. Put $$R = Q^g(\frac{T + T^{1/3 + \eps}}{2\pi  y}) \asymp \frac{Q^g T}{y}.$$ Using the Cauchy-Schwarz inequality and Lemma~\ref{rholambdalemma}, we therefore have
\begin{equation}\label{keyeqfourier}\begin{split}
|\phi(g n(x)& a(y) )|^2 \ll_\eps (NT)^\eps e^{\pi t} \left(\frac{y}{Q^g}\right)   \times \\ &\left(\sum_{0 < n  < R}(|n|, N ^\infty)\left|\lambda_{\pi}\left(\frac{n}{(|n|, N ^\infty)}\right)\right|^2 \right)\left(\sum_{0 < n  < R}\left|\lambda_{\pi_{N}}(n; g)K_{it}\left(\frac{2 \pi |n|y }{Q^g}\right)\right|^2 \right)\end{split}
\end{equation}

\begin{lemma}\label{lambdapinlemma}The function $\lambda_{\pi_{N}}(n; g)$ satisfies the following properties.
\begin{enumerate}
\item Suppose that $n_1$ is a positive integer such that $n_1 |N ^\infty$, and $n_0, n_0'$ are two integers coprime to $N$ such that $n_0 \equiv n_0' \pmod{N_0^g }$. Then $$|\lambda_{\pi_{N
    }}(n_0n_1; g)| = |\lambda_{\pi_{N}}(n_0'n_1; g)| .$$

\item For any integer $r$, and any $n_1 |N ^\infty$, $$\sum_{\substack{r N_0^g  \le |n_0| < (r+1)N_0^g  \\ (n_0, N)=1}} |\lambda_{\pi_{N}}(n_0n_1; g)|^2  \ll N_0^g  n_1^{-1/2}.$$
\end{enumerate}

\end{lemma}
\begin{proof}Let $p|N$ and $u_1, u_2 \in \Z_p^\times$. Then by~\eqref{e:coset}, it follows that for all $w \in \Q_p^\times$, $$|W_{\pi_p}\left(a(wu_1)g_p\right)| = |W_{\pi_p}\left(a(wu_2)g_p\right)|$$ whenever $u_1 \equiv u_2 \mod(p^{n_0(g_p)}).$ It follows that if $n_1| N ^\infty$, then \begin{equation}\label{whitinv}|\lambda_{\pi_{N}}(n_0n_1; g)| = |\lambda_{\pi_{N}}(n_0'n_1; g)| \qquad \text{ if } \ n_0 \equiv n_0' \pmod{N_0^g }.\end{equation}

Furthermore using the above and the Chinese remainder theorem, $$\frac{1}{ N_0^g }\sum_{\substack{n_0 \bmod N_0^g   \\ (n_0, N)=1}} |\lambda_{\pi_{N}}(n_0n_1; g)|^2 = \prod_{p|N} \left(\int_{\Z_p^\times} \left|W_{\pi_p}\left(a(n_1 vp^{-q(g_p)})g_p\right)\right|^2 d^\times v\right)$$ and hence by Proposition~\ref{localwsupportfinal}, $\frac{1}{ N_0^g }\sum_{\substack{n_0 \bmod N_0^g   \\ (n_0, N)=1}} |\lambda_{\pi_{N}}(n_0n_1; g)|^2   \ll n_1^{-1/2}.$
\end{proof}
\begin{lemma}\label{l:firstbound}We have $$\sum_{0 < n  < R}t e^{\pi t}\left|\lambda_{\pi_{N}}(n; g)K_{it}\left(\frac{2 \pi |n|y }{Q^g}\right)\right|^2 \ll_{\epsilon}(NT)^\eps \left(N_0^g  T^{1/3} + \frac{Q^gT}{y}  \right).$$
\end{lemma}
\begin{remark}If we assume Conjecture~\ref{locallindelof}, then the bound on the right side can be improved to $(NT)^\eps \left( T^{1/3} + \frac{Q^gT}{y}  \right).$

\end{remark}

\begin{proof}Let $f(y) = \min(T^{1/3}, \left| \frac{y}{T} - 1\right|^{-1/2}).$ Then it is known  that $t e^{\pi t} |K_{it}(y)|^2 \ll f(y)$; see, e.g., \cite[(3.1)]{templier-sup-2}. Using the previous lemma, we may write \begin{align*}&t e^{\pi t}\sum_{0 < n  < R}\left|\lambda_{\pi_{N}}(n; g)K_{it}\left(\frac{2 \pi |n|y }{Q^g}\right)\right|^2  \ll  \sum_{\substack{1 \le n_1 \le R \\ n_1 |N ^\infty}} \sum_{\substack{1 \le |n_0| \le \frac{R}{n_1} \\ (n_0, N)=1}}  |\lambda_{\pi_{N}}(n_0n_1; g)|^2f\left(\frac{2 \pi |n_0n_1|y }{Q^g}\right)\\ &\ll  \sum_{\substack{1 \le n_1 \le R \\ n_1 |N ^\infty}}  \sum_{0 \le r \le \lfloor \frac{R}{n_1 N_0^g}\rfloor} \sum_{\substack{r N_0^g \le |n_0| \le (r+1) N_0^g  \\ (n_0,
N)=1}}|\lambda_{\pi_{N}}(n_0n_1; g)|^2f\left(\frac{2 \pi |n_0n_1|y }{Q^g}\right)\\ &\ll N_0^g \sum_{\substack{1 \le n_1 \le R \\ n_1 |N ^\infty}} n_1^{-1/2} \sum_{0 \le r \le \lfloor\frac{R}{n_1 N_0^g}\rfloor} f\left(\frac{2 \pi |n_0^{(r)}n_1|y }{Q^g}\right)  \quad \text{ \big[where } n_0^{(r)} \in [r N_0^g, (r+1) N_0^g] \\ &\text{is the point where } f(2 \pi n_0^{(r)} n_1 y/ Q^g) \text{ is maximum, and where we have used Lemma \ref{lambdapinlemma}\big]}\\
&\ll \sum_{\substack{1 \le n_1 \le R \\ n_1 |N ^\infty}} n_1^{-1/2}N_0^g \left( T^{1/3} + \int_0^{\frac{R}{N_0^g n_1}} f\left(\frac{2 \pi N_0^g r n_1y }{Q^g}\right) dr \right) \quad \text{[as $f$ has $\ll 1$ turning points]} \\&\ll \sum_{\substack{1 \le n_1 \le R \\ n_1 |N ^\infty}} \left( N_0^g  T^{1/3}n_1^{-1/2} + n_1^{-3/2} \frac{Q^g}{y}\int_0^{T+T^{1/3 + \eps}} \left| \frac{s}{T} - 1\right|^{-1/2} ds \right) \\&\ll_\eps (NT)^\eps \left(N_0^g  T^{1/3} + \frac{Q^gT}{y}  \right).
\end{align*}
\end{proof}

\begin{lemma}\label{l:secondbd}For all $X >0$, we have
$$\sum_{0 < n  < X}(|n|, N ^\infty)\left|\lambda_{\pi}\left(\frac{n}{(|n|, N ^\infty)}\right)\right|^2\ll_{\epsilon}  (NTX)^\epsilon X.$$

\end{lemma}
\begin{proof}This follows from the last part of Lemma~\ref{rholambdalemma} using a similar (but simpler) argument as in the above Lemma.
\end{proof}

Finally, by combining~\eqref{keyeqfourier}, Lemma~\ref{l:firstbound}, and Lemma~\ref{l:secondbd}, we get the bound
\begin{equation}\label{strongerfourierbd}|\phi(g n(x) a(y) )|^2 \ll_\eps (NT)^\eps \left(\frac{Q^gT}{y} + N_0^g T^{1/3} \right).\end{equation} Taking square roots, and using that $ Q^g \le N_0M_1$, $N_0^g \le N_0$, we get the conclusion of Proposition~\ref{fourierboundprop}.

\subsection{Preliminaries on amplification}\label{s:amplglobal}

Our aim for the rest of this paper is to prove the following proposition. As explained in Section \ref{s:proofsketch}, this will complete the proof of our main result.

\begin{proposition}\label{prop:ampl} Let $\Lambda \ge 1$ be a real number. Let $n(x)a(y) \in \F_{N_2}$, $g  \in \J_N $. Then \begin{equation}\label{finalamplbdnew}
|\phi(g n(x)a(y))|^2 \ll_\eps (\Lambda NT)^\eps \ N_1 M_1 \left[\frac{T + N_2^{1/2}T^{1/2}y}{\Lambda}   + \Lambda^{1/2}T^{1/2}(N_2^{-1/2} + y)  + \Lambda^2 T^{1/2}N_2^{-1}  \right].   \end{equation}

\end{proposition}
\bigskip

Recall that $h_{N} = \prod_{p|N} a(p^{n_{1,p}}).$ Define the vector $\phi' \in V_\pi$ by $$\phi'(g) = \phi(gh_{N}).$$ Then the problem becomes equivalent to bound the quantity $\phi'(k_{N} n(x)a(y))$ where $k_{N} \in K_{N}= \prod_{p|N}G(\Z_p)$, $k_Nh_N \in \J_N$. Note that $\phi'$ is $K_1'(N)K_\infty$-invariant where $K_1'(N):=h_{N}K_1(N)h_{N}^{-1}$.

Define the function $\Phi'_N$ on $\prod_{p|N}G(\Q_p)$ by $\Phi'_N = \prod_{p|N}\Phi'_{\pi_p}$, with the functions $\Phi'_{\pi_p}$ defined in Section~\ref{s:mcbeg}. By Proposition~\ref{keymatrixprop}, it follows that  $$R(\Phi'_N)\phi' :=\int_{(Z\bs G)(\prod_{p|N}\Q_p)} \Phi'_N(g)(\pi(g)\phi')  \ dg = \delta_N \phi'$$ where $\delta_N \gg N_1^{-1}M_1^{-1}.$ Note also that if $g \in \prod_{p|N}G(\Q_p)$ and $\Phi'_N(g)\neq 0$, then $g \in Z(\Q_p)G(\Z_p)$ for each prime $p$ dividing $N$ and $g \in Z(\Q_p)K^0_p(p)$ for each prime $p$ dividing $N_2$. Also, recall that $R(\Phi'_N)$ is a self-adjoint, essentially idempotent operator.

Next, we consider the primes not dividing $N$. Let $\mathcal{H}_{\ur}$ be the usual global (unramified) convolution Hecke algebra; it is generated by the set of all
functions $\kappa_\ur$ on $\prod_{p \nmid N}G(\Q_p)$ such that for each finite prime $p$ not dividing $N$, \begin{enumerate}

\item $\kappa_p \in C_c^\infty(G(\Q_p), \omega_{\pi_p}^{-1})$ ,

\item $\kappa_p$ is bi-$G(\Z_p)$ invariant.

\end{enumerate}

It is well-known that $\mathcal{H}_{\ur}$ is a commutative algebra and is generated by the various functions $\kappa_\ell$ (as $\ell$ varies over integers coprime to $N$)  where $\kappa_\ell = \prod_{p \nmid N} \kappa_{\ell,p}$ and the function $\kappa_{\ell,p}$ in $C_c^\infty(G(\Q_p), \omega_{\pi_p}^{-1})$ is defined as follows:
\begin{enumerate}
\item $\kappa_{\ell,p}(zka(\ell)k) = |\ell|^{-1/2} \omega_{\pi_p}^{-1}(z)$ for all $z \in Z(\Q_p)$, $k\in G(\Z_p)$.
\item $\kappa_{\ell,p}(g)=0$ if $g \notin Z(\Q_p)G(\Z_p)a(\ell)G(\Z_p)$.
\end{enumerate}

Then, it follows that for each $\kappa_{\ur} \in \mathcal{H}_{\ur}$, $$R(\kappa_{\ur})\phi' :=\int_{\prod_{p\nmid N}(Z\bs G)(\Q_p)} \kappa_{\ur}(g)(\pi(g)\phi')  \ dg = \delta_\ur \phi'$$ where $\delta_\ur$ is a complex number (depending linearly on $\kappa_{\ur}$).  Furthermore, $$R(\kappa_\ell) \phi' = \lambda_\pi(\ell) \phi'$$ where the Hecke eigenvalues $\lambda_\pi(\ell)$ were defined earlier in Lemma~\ref{rholambdalemma}. Moreover, we note that as $\kappa_{\ur}$  varies over $\mathcal{H}_{\ur}$, the corresponding operators $R(\kappa_{\ur})$ form a commuting system of normal operators. Indeed, if we define $\kappa_\ell^* = \left(\prod_{p|\ell}\omega_{\pi_p}^{-1}(\ell)\right) \kappa_\ell,$  and extend this via multiplicativity and anti-linearity to all of $\mathcal{H}_{\ur}$, then we have an involution $\kappa\mapsto \kappa^*$ on all of $\mathcal{H}_{\ur}$. It is well-known that $\kappa^*(g) = \overline{\kappa(g^{-1})}$ and hence $R(\kappa^*)$ is precisely the adjoint of $R(\kappa)$.

Finally, we consider the infinite place. For $g\in G(\R)^+$, let $u(g) = \frac{|g(i) - i|^2 }{4 \Im(g(i))}$ denote
the hyperbolic distance from $g(i)$ to $i$. Each bi-$Z(\R)K_\infty$-invariant function $\kappa_\infty $ in $C_c^\infty(Z(\R)\bs G(\R)^+)$, can be viewed as a function on $\R^+$ via $\kappa_\infty(g) = \kappa_\infty(u(g))$. For each irreducible spherical unitary principal series representation $\sigma$ of $G(\R)$, we  define the Harish-Chandra--Selberg transform $\hat{\kappa}_\infty(\sigma)$ via
$$\hat{\kappa}_\infty(\sigma) = \int_{Z(\R) \bs G(\R)^+} \kappa_\infty(g)\frac{ \langle \sigma(g)v_\sigma, v_\sigma \rangle}{\langle v_\sigma, v_\sigma \rangle}dg $$
where $v_\sigma$ is the unique (up to multiples) spherical vector in the
representation $\sigma$. It is known that for all such $\sigma$,  $R(\kappa_\infty) v_\sigma = \hat{\kappa}_\infty(\sigma)v_\sigma$; in particular,  $R(\kappa_\infty) \phi' = \hat{\kappa}_\infty(\pi_\infty)\phi'.$

By~\cite[Lemma 2.1]{templier-sup-2} there exists such a function $\kappa_\infty$ on $G(\R)$ with
the following properties:

\begin{enumerate}

\item $\kappa_\infty(g) = 0$ unless $g \in G(\R)^+$ and $u(g)\le 1$.

\item $\hat{\kappa}_\infty(\sigma) \ge 0$ for all irreducible spherical unitary principal series representations $\sigma$ of $G(\R)$.

\item    $\hat{\kappa}_\infty(\pi_\infty) \gg 1$.

\item  For all $g \in G(\R)^+$, $|\kappa_\infty(g)| \le T$ and moreover, if
$u(g) \ge T^{-2} $, then  $|\kappa_\infty(g)| \le \frac{T^{1/2}}{u(g)^{1/4}}$.

\end{enumerate}

Henceforth, we fix a function $\kappa_\infty$ as above.

\subsection{The amplified pre-trace formula}
In this subsection, we will use $L^2(X)$ as a shorthand for $L^2(G(\Q)\bs G(\A) / K_1'(N)K_\infty ,  \ \omega_\pi)$.

Let the functions $\Phi_N'$, $\kappa_\infty$ be as defined in the previous subsection. Consider the space of functions $\kappa$ on $G(\A)$ such that $\kappa = \Phi_N' \kappa_{\ur} \kappa_\infty$ with $\kappa_{\ur}$ in $\mathcal{H}_{\ur}$. We fix an orthonormal basis $\mathcal{B} = \{\psi\}$ of the space  $L^2(X)$ with the following properties:

\begin{itemize}

\item $\phi' \in \mathcal{B}$,

\item Each element of $\mathcal{B}$ is an eigenfunction for all the operators $R(\kappa)$ with $\kappa$ as above, i.e., for all $\psi \in \mathcal{B}$, there exists a complex number $\lambda_\psi$ satisfying $$  R(\Phi_N')R(\kappa_{\ur})R(\kappa_\infty)\psi = R(\kappa)\psi := \int_{Z(\A) \bs G(\A)} \kappa(g)(\pi(g)\psi)  \ dg = \lambda_\psi \psi.$$

\end{itemize}
 Such a basis exists because the set of all $R(\kappa)$ as above form a commuting system of normal operators. The basis $\mathcal{B}$ naturally splits into a discrete and continuous part, with the continuous part consisting of Eisenstein series and the discrete part consisting of cusp forms and residual functions.

Given a $\kappa = \Phi_N' \kappa_{\ur} \kappa_\infty$ as above, we define the automorphic kernel $K_\kappa(g_1, g_2)$ for $g_1, g_2 \in G(\A)$
via $$K_\kappa(g_1, g_2) = \sum_{\gamma \in Z(\Q) \bs G(\Q)} \kappa(g_1^{-1} \gamma
g_2).$$

A standard calculation tells us that if $\psi=\otimes_v \psi_v$ is an element of $L^2(X)$ such that for each place $v$, $\psi_v$ is an eigenfunction for $R(\kappa_v)$ with eigenvalue $\lambda_v$, then
one has

\begin{equation}\label{RK1} \int_{Z(\A)G(\Q) \bs G(\A)}K_\kappa(g_1, g_2)\psi(g_2) dg_2 = (\prod_v \lambda_v)\psi(g_1) \end{equation}

\begin{lemma}Suppose that $\kappa_\ur = \kappa'_\ur \ast (\kappa'_\ur)^*$ for some $\kappa'_\ur \in \mathcal{H}_{\ur}$. Put $\kappa = \Phi_N' \kappa_{\ur} \kappa_\infty$. Then if $\psi \in  \mathcal{B}$ then $$\int_{Z(\A)G(\Q) \bs G(\A)}K_\kappa( g_1  , g_2)\psi(g_2) dg_2 = \lambda_\psi \psi(g_1)$$ for some $\lambda_\psi \ge 0$. Moreover $\lambda_{\phi'} \ge M_1^{-1}N_1^{-1} |\lambda'_{\ur}|^2 \hat{\kappa}_\infty(\pi_\infty) $ where the quantity $\lambda'_{\ur}$ is defined by $R(\kappa'_\ur)\phi' = \lambda'_{\ur}\phi'$.

\end{lemma}
\begin{proof}By our assumption that $\psi \in  \mathcal{B}$, a complex number $\lambda_\psi$ as above exists. We can write $\lambda_\psi = \lambda_{\psi,N} \lambda_{\psi,\ur} \lambda_{\psi,\infty}$ using the decomposition  $R(\kappa) = R(\Phi_N')R(\kappa_{\ur})R(\kappa_\infty)$. We have $\lambda_{\psi,\infty} \ge 0$ by our assumption $\hat{\kappa}_\infty(\sigma) \ge 0$ for all irreducible spherical unitary principal series representations $\sigma$ of $G(\R)$. We have $\lambda_{\psi,N} \ge 0$ by Corollary~\ref{nonneg}. Finally if $R(\kappa'_\ur) \psi = \lambda_\psi' \psi$ then $\lambda_{\psi,\ur} = |\lambda_\psi'|^2 \ge 0$. Hence $\lambda_\psi \ge 0$. The last assertion is immediate from the results of the previous subsection.

\end{proof}

Henceforth we assume that $\kappa_\ur = \kappa'_\ur \ast (\kappa'_\ur)^*$ for some $\kappa'_\ur \in \mathcal{H}_{\ur}$ and we put $\kappa = \Phi_N' \kappa_{\ur} \kappa_\infty$. Then spectrally expanding $K_\kappa( g  , g)$ along $\mathcal{B}$ and using the above Lemma, we get for all $g \in G(\A)$,

$$ M_1^{-1}N_1^{-1} \hat{\kappa}_\infty(\pi_\infty)|\lambda'_{\ur} \ \phi'(g)|^2  \le K_\kappa(g,g).$$

     Note that $\hat{\kappa}_\infty(\pi_\infty) \ge 1$. Next we look at the quantity $K_\kappa(g,g)$. Assume that $g = k_{N} n(x)a(y)$ with $k_{N} = \prod_{p|N}k_p \in K_{N}$, $k_Nh_N \in \J_N$. The second condition means that $k_p \in w K_p^0(p)$ for all $p|N_2$. We have $$K_\kappa(g,g) = \sum_{\gamma \in Z(\Q) \bs G(\Q)}\Phi'_N(k_{N}^{-1}\gamma k_{N}) \kappa_\ur(\gamma) \kappa_\infty((n(x)a(y))^{-1}\gamma n(x)a(y)).$$  Above we have $\Phi'_N(k_{N}^{-1}\gamma k_{N} ) \le 1$, and moreover if $\Phi'_N(k_{N}^{-1}\gamma k_{N} ) \neq 0$ then we must have a) $k_{p}^{-1}\gamma k_{p} \in Z(\Q_p)G(\Z_p)$ for all primes $p$ dividing $N$, and b) $k_{p}^{-1}\gamma k_{p} \in Z(\Q_p) K_p^0(p)$ for all primes $p$ dividing $N_2$. The condition a) implies that $\gamma \in Z(\Q_p)G(\Z_p)$ for all primes $p$ dividing $N$. The  condition b), together with the fact that $k_p \in w K_p^0(p)$ for all $p|N_2$,  implies that $\gamma \in Z(\Q_p)K_{0,p}(p)$ for all primes $p$ dividing $N_2$.

     Finally we have $\kappa_\infty(g) = 0$ if $\det(g)<0$, and if $\det(g)>0$ we can write $\kappa_\infty(g) = \kappa_\infty(u(g))$ as explained earlier, whence
     $$\kappa_\infty((n(x)a(y))^{-1}\gamma n(x)a(y)) =\kappa_\infty(u(z, \gamma z)), \quad z=x+iy$$ where for any two points $z_1, z_2$ on the upper-half plane, $u(z_1, z_2)$ denotes the hyperbolic distance between them, i.e., $u(z_1, z_2) = \frac{|z_1 - z_2|^2 }{4 \Im(z_1)\Im(z_2)}.$
Putting everything together, we get the following Proposition.

     \begin{proposition}\label{p:ampl}Let $\kappa'_\ur \in \mathcal{H}_{\ur}$ and suppose that $R(\kappa'_\ur)\phi' = \lambda'_\ur \phi'$. Let $\kappa_\ur = \kappa'_\ur \ast (\kappa'_\ur)^*$ and  $\kappa = \Phi_N' \kappa_{\ur} \kappa_\infty$. Then for all $z=x+iy$  and all $k \in K_{N}$ such that $kh_N \in \J_N$, we have

$$|\phi'(k n(x)a(y))|^2 \le \frac{M_1N_1}{|\lambda'_\ur|^2}  \sum_{\substack{\gamma \in Z(\Q) \bs G(\Q)^+, \\ \gamma \in Z(\Q_p)K_{0,p}(p) \forall p|N_2 \\ \gamma  \in Z(\Q_p)G(\Z_p) \forall p|N  }} |\kappa_\ur(\gamma) \ \kappa_\infty(u(z, \gamma z))|  $$

\end{proposition}

\subsection{Conclusion}\label{s:conclusion}

We now make a specific choice for $\kappa_\ur$. Let $\Lambda \ge 1$ be a real number. We let $$S= \{\ell: \ell \text{ prime, }(\ell, N)=1, \ \Lambda \le \ell \le 2\Lambda\}. $$

Define for each integer $r$, $$c_r = \begin{cases} \frac{|\lambda_\pi(r)|}{\lambda_\pi(r)} &\text{ if } r = \ell \text{ or } r=\ell^2, \ \ell \in S, \\ 0 &\text{ otherwise.} \end{cases}$$

We put $\kappa'_\ur = \sum_r c_r \kappa_r$, and $\kappa_\ur = \kappa'_\ur \ast (\kappa'_\ur)^*$. Given this, let us estimate the quantities appearing in Proposition~\ref{p:ampl}.

First of all, we have $\lambda'_\ur = \sum_{\ell \in S}(|\lambda_\pi(\ell)| + |\lambda_\pi(\ell^2)|).$ By the well-known relation $\lambda_\pi(\ell)^2 - \lambda_\pi(\ell^2) = \omega_{\pi_\ell}(\ell)$, it follows that $|\lambda_\pi(\ell)| + |\lambda_\pi(\ell^2)| \ge 1$. Hence $\lambda'_\ur \gg_\eps \Lambda^{1-\eps}$.

Next, using the well-known relation $$\kappa_m \ast \kappa_n^\ast = \sum_{t|
\gcd(m,n)} \left(\prod_{p|t}\omega_{\pi_p}(t)\right) \left(\prod_{p|n}\omega_{\pi_p}^{-1}(n)\right) \kappa_{mn/t^2},$$

we see that $$\kappa_\ur = \sum_{1 \le l \le 16\Lambda^4} y_l \kappa_l$$ where the complex numbers $y_l$ satisfy:

$$|y_l| \ll \begin{cases}\Lambda, & l=1,\\1, &l = \ell_1 \text{ or } l = \ell_1 \ell_2 \text{ or } l = \ell_1\ell_2^2 \text{ or } l = \ell_1^2\ell_2^2 \text{ with } \ell_1, \ell_2 \in S \\ 0, &\text{otherwise.}\end{cases}$$

We have $|\kappa_l(\gamma)| \le l^{-1/2}$. Moreover $\kappa_l(\gamma) = 0$ unless $\gamma \in Z(\Q_p)G(\Z_p)a(\ell)G(\Z_p)$ for all $p \nmid N$. We deduce the following bound,

\begin{equation}\label{e:29}|\phi'(k n(x)a(y))|^2 \ll_\eps  \Lambda^{-2+\eps} M_1N_1  \sum_{1 \le l \le 16\Lambda^4} \frac{y_l}{\sqrt{l}}\sum_{\substack{\gamma \in Z(\Q) \bs G(\Q)^+, \\ \gamma \in Z(\Q_p)K_{0,p}(p) \forall p|N_2 \\ \gamma  \in Z(\Q_p)G(\Z_p)a(\ell)G(\Z_p) \forall p\nmid N_2  }} | \kappa_\infty(u(z, \gamma z))| . \end{equation}

Define $$M(\ell, N_2) = \{\mat{a}{b}{c}{d}, \ a,b,c,d \in \Z, \ a>0, N_2 |c,  \ ad-bc= \ell\}.$$

The following lemma follows immediately from strong approximation.

\begin{lemma}Let $\gamma \in G(\Q)^+$ and $\ell$ be a positive integer coprime to $N_2$. Suppose that for each prime $p$, $\gamma \in Z(\Q_p)G(\Z_p)a(\ell)G(\Z_p)$. Suppose also that for each prime $p|N_2$, $\gamma \in Z(\Q_p)K_{0,p}(p)$.
Then there exists $z \in Z(\Q)$ such $z\gamma \in M(\ell, N_2)$.
\end{lemma}
\begin{proof} Omitted.
\end{proof}

Let us take another look at \eqref{e:29} in view of the above Lemma. The sum in \eqref{e:29} is over all matrices $\gamma$ in $Z(\Q) \bs G(\Q)^+$ such that  $\gamma \in Z(\Q_p)K_{0,p}(p)$ for $p|N_2$ and $\gamma  \in Z(\Q_p)G(\Z_p)a(\ell)G(\Z_p)$ for $p\nmid N_2$. The latter condition can equally well be taken to over all $p$ as $Z(\Q_p)G(\Z_p)a(\ell)G(\Z_p) = Z(\Q_p)G(\Z_p)a(\ell)G(\Z_p)= Z(\Q_p)G(\Z_p)$ if $\ell$ and $p$ are coprime, which is certainly the case when $p|N$. Therefore the above Lemma, together with the fact  that the natural map from $M(\ell, N_2)$ to $Z(\Q) \bs G(\Q)^+$ is an injection, implies that the sum in \eqref{e:29} can replaced by a sum over the set  $M(\ell, N_2)$. Hence, writing $g = kh_{N} \in \J_N $ as before, we get

\begin{equation}\label{usefulsum}|\phi(g n(x)a(y))|^2 = |\phi'(k n(x)a(y))|^2\ll_\eps  \Lambda^{-2+\eps} M_1N_1  \sum_{1 \le l \le 16\Lambda^4} \frac{y_l}{\sqrt{l}}\sum_{\gamma \in M(\ell, N_2)} | \kappa_\infty(u(z, \gamma z))| . \end{equation}

For any $\delta>0$, we define

 $$N(z, \ell, \delta, N_2) = \left| \{\gamma \in M(\ell, N_2) : u(z, \gamma z) \le \delta\}\right|.$$
We have the following counting result due to Templier~\cite[Proposition 6.1]{templier-sup-2}.

\begin{proposition}\label{countingresult}Let $z=x+iy \in \F_{N_2}$. For any $0< \delta<1$ and a positive integer $\ell$ coprime to $N_2$, let the number $N(z, \ell, \delta, N_2)$ be defined as above.

For $\Lambda \ge 1$, define $$A(z, \Lambda, \delta, N_2) = \sum_{1 \le l \le 16\Lambda^4} \frac{y_l}{\sqrt{l}} N(z, \ell, \delta, N_2).$$ Then

$$A(z, \Lambda, \delta, N_2)\ll_\eps \Lambda ^\eps N_2^\eps \bigg[\Lambda   + \Lambda N_2^{1/2} \delta^{1/2} y + \Lambda^{5/2}\delta^{1/2}N_2^{-1/2}+ \Lambda^{5/2}\delta^{1/2}y  +  \Lambda^4  \delta N_2^{-1}\bigg]. $$

\end{proposition}
\begin{proof}This is just Prop. 6.1 of \cite{templier-sup-2}.
\end{proof}
 Now~\eqref{usefulsum} gives us

\begin{equation}\label{usefulsum2}|\phi(g n(x)a(y))|^2 \ll_\eps  \Lambda^{-2+\eps} M_1N_1  \int_{0}^1 |\kappa_\infty(\delta)|  dA(z, \Lambda, \delta, N_2). \end{equation}

Using Proposition~\ref{countingresult} and the property $|\kappa_\infty(\delta)| \le \min(T,\frac{T^{1/2}}{\delta^{1/4}})$, we immediately deduce Proposition~\ref{prop:ampl} after a simple integration, as in \cite[6.2]{templier-sup-2}.

\bibliography{refs-que}

\end{document}